\documentclass[a4paper,11pt]{amsart}
\usepackage[all]{xy}
\usepackage{amssymb}
\usepackage{amsmath}
\usepackage{graphicx}

\newtheorem{definition}{Definition}[section]
\newtheorem{lemma}[definition]{Lemma}
\newtheorem{prop}[definition]{Proposition}
\newtheorem{cor}[definition]{Corollary}

\newtheorem{remark}[definition]{Remark}
\newtheorem{theorem}[definition]{Theorem}
\theoremstyle{definition}\newtheorem{exam}[definition]{Example}
\newtheorem{fact}[definition]{Fact}

\newcommand{\Disc}{\mathop{{\rm Disc}}\nolimits}
\newcommand{\pr}{\perp}

\newcommand{\Tych}{\mathop{{\rm Tych}}\nolimits}

\newcommand{\ov}{\overline}

\newcommand{\pty}{\emptyset}
\newcommand{\ots}{\otimes}
\newcommand{\sep}{\supseteq}
\newcommand{\seq}{\subseteq}

\newcommand{\sm}{\setminus}
\newcommand{\ts}{\times}
\newcommand{\wg}{\wedge}
\newcommand{\wh}{\widehat}

\newcommand{\lan}{\langle}
\newcommand{\ran}{\rangle}

\newcommand{\La}{\Leftarrow}
\newcommand{\Ra}{\Rightarrow}

\newcommand{\al}{\alpha}

\newcommand{\del}{\delta}
\newcommand{\Del}{\Delta}

\newcommand{\ga}{\gamma}

\newcommand{\om}{\omega}
\renewcommand{\phi}{\varphi}

\renewcommand{\th}{\theta}
\newcommand{\Th}{\Theta}

\newcommand{\bpm}{\begin{pmatrix}}
\newcommand{\epm}{\end{pmatrix}}
\newcommand{\bsm}{\begin{smallmatrix}}
\newcommand{\esm}{\end{smallmatrix}}

\newcommand{\cF}{\mathcal{F}}
\newcommand{\cL}{\mathcal{L}}
\newcommand{\cP}{\mathcal{P}}
\newcommand{\cT}{\mathcal{T}}
\newcommand{\cZ}{\mathcal{Z}}

\newcommand{\fF}{\mathfrak{F}}
\newcommand{\fG}{\mathfrak{G}}
\newcommand{\fH}{\mathfrak{H}}

\newcommand{\BG}{\mathbb{G}}

\newcommand{\N}{\mathbb{N}}
\newcommand{\Q}{\mathbb{Q}}
\newcommand{\R}{\mathbb{R}}
\newcommand{\Z}{\mathbb{Z}}

\newcommand{\lt}{\left}
\newcommand{\rt}{\right}

\DeclareMathOperator{\CA}{CoAcc}%-------------------------------
\DeclareMathOperator{\CD}{CoDcc}%-------------------------------
\DeclareMathOperator{\Cof}{Cof}%-------------------------------
\DeclareMathOperator{\FU}{Funct} %----------------
\DeclareMathOperator{\Map}{Map} %----------------

\DeclareMathOperator{\supp}{supp}%-------------------------------
\theoremstyle{definition}

\def\ga{\gamma}
\def\al{\alpha}
\def\ph{\varphi}

\def\ba{\mathbf a}
\def\bb{\mathbf b}
\def\bc{\mathbf c}

\newcommand{\set}[1]{\left\{#1\right\}}
\newcommand{\liml}{\lim\limits}
\newcommand{\suml}{\sum\limits}
\newcommand{\LR}{\Leftrightarrow}

\def\Ta{\mathcal{T}}
\def\Za{\mathcal{Z}}
\def\La{\mathcal{L}}

\def\fG{\mathfrak G}
\def\fF{\mathfrak F}
\def\fH{\mathfrak H}

\def\th{{\mathrm{Tych}}}

\def\Cof{\mathrm{Cof}\,}

\def\MOG{\mathcal{M}_0(\mathfrak{G})}
\def\MG{\mathcal{M}(\fG)}

\begin{document}

\footskip30pt

\date{}

\title{Topological linear spaces of formal linear sums and continuous linear operators}

\author{Nikolay I. Dubrovin}

\address[N.~Dubrovin]{Department of Algebra and geometry, Vladimir State University,
Gorky Street  87, Vladimir, 600000, Russia}
\email{ndubrovin81@gmail.com; phone +79157992425}
\thanks{The author is supported by the EPSRC grant no. EP/D077907/1.
He also thanks Manchester University for kind hospitality.}


\begin{abstract}
The rings of linear continuous operators on the topological spaces
of $\fG$-zero maps were described, where $\fG$ is a filter on a set with an involution.
This applies to modules of formal series with well ordered support over left
ordered groups.
\end{abstract}

\maketitle

\pagestyle{plain}

\section{Introduction}\label{S-Int}

If $X$ is a topological space, then the ring $C(X)$ of continuous functions from
$X$ to $\R$ is a classical object in topology and analysis. For instance, one can be
interested in properties of $C(X)$ as a ring, and from  this point of view the situation
is well understood (see \cite{G-J}).

If $X$ has an additional structure, for instance, if $X$ is a linear topological space,
one can consider the properties of the ring of linear continuous functions from $X$
to $\R$ endowing it with different `natural' topologies (see \cite{Bou}). However
sometimes we have to deal with the case, when the target, $K$, of maps from $X$ is a
noncommutative skew field. For instance, this situation occurs trying to embed a group ring
$FG$ of a (torsion-free) group $G$ over a field $F$ into a skew field.

For example, let $\BG$ be the universal covering of the group $SL_2(\R)$ and
$U=\bigl\{\lt(\bsm a&b\\0&a^{-1}\esm\rt)\mid a, b\in \R,\, a>0 \bigl\}$ is a subgroup of
$SL_2(\R)$. Since $U$ is metabelian and torsion-free, the group ring $FU$ (over any field $F$)
is an Ore domain, therefore its classical quotient field $K$ is a (noncommutative) skew field.
Trying to extend this
embedding to an embedding of $K \BG$ into a skew field, the first author developed the
following approach (see \cite{Du94}). He considered the space $K\{\CD \BG\}$ of all formal
series on $\BG$ with well ordered support, and the group ring $K\BG$ acting on this space by left
multiplication. Then one can invert the elements of $K \BG$ as linear maps forming a
rational closure $D$ of $K\BG$ in $K\{\CD \BG\}$. The behavior of elements of $D$ is
quite complicated, and the aforementioned paper contains a series of algebraic conditions on
elements of $D$. It is quite difficult and tedious to verify that these conditions
respect basic operations. Later (see \cite{DuDu2}) Dubrovin noticed that an essential part of the
proof can be simplified by endowing  $K\{\CD \BG\}$ with a structure of a linear topological
space such that elements of $D$ become linear continuous maps. Thus the aforementioned
algebraic conditions can be understood as well known properties of continuous maps.

In this paper we develop a very general approach to tackling this situation. Namely,
with each filter $\fG$ on a set $G$ we connect a linear space $ K \{\fG\}$ of all maps from
$G$ to $K$ whose zero set belongs to $\fG$. We endow this space with a linear Hausdorff
topology making it into linear topological space.
The examples of such topologies
include Tychonoff topology on the product of spaces, but also the adic topology on the space
of Laurent power series. However, the example of our main interest will be the space
of formal series $K\{\CD G\}$ with well ordered support over a left ordered group $G$.
One of the main result of the paper describes linear continuous maps between topological
spaces $K \{\fG\}$ and $K\{H\}$, where $\fG$ and $\fH$ are filters on sets $G$ and $H$
with involution (see Theorem~\ref{thm:endo}). In particular, we completely characterize
such maps in terms of zero sets of their (infinite) matrices. Namely, these zero sets must
belong to a special filter on the direct product of $H$ and $G$, which were
introduced and investigated in \cite{DuDu1}.

As a corollary we give a matrix description of the ring of continuous operators of the
space $K \{\fG\}$ (see Theorem~\ref{thm:Ring}).

There is a different approach how to embed a group ring of a countable torsion-free group
into a skew field, based on the theory of $C^*$-algebras and operators on Hilbert spaces
(see \cite{Luck}). From this point of view this paper is a first step in developing a similar
machinery in a more general and abstract situation. For instance, in Section~\ref{S-Pairing}
we introduce the operation of pairing on formal sums which resembles scalar product in
Hilbert spaces.

\section{Filters}\label{S-Fil}

In this section we recall some basic facts and definitions, and also
some results from \cite{DuDu1}.

Let $G$ be a set. A nonempty collection  $\fG$ of subsets of $G$ is
said to be a \emph{filter}, if it is closed with respect to finite
intersections and supersets. For instance, we allow the set of all
subsets of $G$, $\cP(G)$, to be a filter. Clearly, if $\fG$ is a
filter, then $\fG=\cP(G)$ iff $\pty\in \fG$. If $\fG\neq \cP(G)$,
then $\fG$ is said to be a \emph{proper filter} on $G$.

Let $\cL$ be a collection of subsets of $G$ with the following
property: for all $A, B\in \cL$ there is $C\in \cL$ such that $C\seq
A\cap B$. Then $\fG=\{B\mid A\seq B$ for some $A\in \cL\}$ is a
\emph{filter generated by $\cL$}, and $\cL$ is a \emph{filter base
for $\fG$}.

%The following examples of filters will play a prominent role in this paper. 
We say that a subset $A$ of $G$ is \emph{cofinite}, if its
complement $\ov A$ is a finite set. The \emph{Frechet filter} on
$G$, $\Cof(G)$, consists of all cofinite subsets of $G$. Clearly
$\Cof(G)$ is a proper filter iff $G$ is an infinite set.

Let $G$ be a linearly ordered set. A subset $\Delta\subseteq G$ is
said to be \emph{well ordered}, if every nonempty subset of $\Delta$
has a minimal element. This is the same as $\Delta$ has a
\emph{descending chain condition} (d.c.c.): every descending chain
of elements $a_1\geq a_2\geq \dots$ of $\Delta$ stabilizes. Clearly
$\Delta$ has a d.c.c. iff it contains no (strictly) descending chain
$a_1> a_2> a_3> \dots$.

Similarly, $\Delta\subseteq G$ has an \emph{ascending chain
condition} (a.c.c.), if every ascending chain $a_1\leq a_2\leq\dots$
of its elements stabilizes. Thus $\Delta$ has an a.c.c. iff it
contains no (strictly) ascending chain $a_1 < a_2 < \dots$ iff $G$
is well ordered in the dual ordering.

Suppose that $(G, \leq)$ is a linearly ordered set. Let $\CD(G)$
denote the collection of all subsets of $G$ whose complement has a
d.c.c. Since the union of two well ordered subsets of $G$ is well
ordered, $\CD(G)$ is a filter on $G$, and it is a proper filter iff
$G$ is not well ordered.

Similarly let $\CA(G)$ be a collection of all subsets of $G$ whose
complement has an a.c.c. Then $\CA(G)$ is a filter on $G$ and this
filter is proper iff $G$ contains a strictly ascending chain.

We can order the filters on $G$ by inclusion: $\fG_1\leq \fG_2$ if
$\fG_1\seq \fG_2$. It is easily checked that with respect to this
ordering the set of all filters on $G$ forms a lattice, that is, for
any filters $\fG_1$ and $\fG_2$ there is a least filter $\fG_1 \vee
\fG_2$ containing  $\fG_1$ and $\fG_2$, and there is a largest
filter $\fG_1 \wg \fG_2$ which is contained in both
 $\fG_1$ and $\fG_2$.

The following remark describes the operations in this lattice.

\begin{remark}\label{oper-filt}
Let $\fG_1$, $\fG_2$ be filters on a set $G$. Then $\fG_1\wg \fG_2$
is given by the intersection of filters: $\fG_1\cap \fG_2=\{A\seq
G\mid A\in \fG_1$ and $A\in \fG_2\}$. Furthermore, $\fG_1\vee \fG_2$
is the filter generated by all intersections $A\cap B$, $A\in
\fG_1$, $B\in \fG_2$.
\end{remark}

We define a new operation on filters. In ring theory this operation
corresponds to the quotient of ideals. Suppose that $\fG_1$ and
$\fG_2$ are filters on $G$. Define $\fG_1: \fG_2=\{A\seq G \mid
A\cup A'\in \fG_1$ for every $A'\in \fG_2\}$.

\begin{fact}[\cite{DuDu1}, property 6]\label{quot-fil}
$\fG_1: \fG_2$ is a filter on $G$. Furthermore, $\fG_1: \fG_2$ is
the largest filter $\cF$ on $G$ with the property $\cF\cap \fG_2\seq
\fG_1$.
\end{fact}

The following remark is straightforward.

\begin{remark}\label{quot}
1) If $\fF_1\seq \fF_2$ then $\fF_1:\fG\seq \fF_2:\fG$.

2) If $\fG_1\seq \fG_2$ then $\fF:\fG_1\sep \fF: \fG_2$.

3) $\fF: \fG=(\fF\cap \fG): \fG$.

4) $\fF \seq \fF: \fG$.
\end{remark}

If $\fG$ is a filter, then define \[\fG^{\perp}=\{A\seq G \mid A\cup
A' \text{ is cofinite for every } A'\in \fG\}.\] For instance,
$\Cof(G)^{\perp}=\cP(G)$ and $\cP(G)^{\perp}=\Cof(G)$. It follows
from the definition that $\fG^{\perp}=\Cof(G): \fG$, hence
$\fG^{\perp}$ is a filter by Fact~\ref{quot-fil}. Furthermore,
Remark~\ref{quot}, 4) implies that $\Cof(G)\seq \fG^{\perp}$.

\begin{lemma}\label{two-perp}
$\fG\seq \fG^{\perp\perp}$ and $\fG^{\perp}=\fG^{\perp\perp\perp}$
for any filter $\fG$.
\end{lemma}
\begin{proof}
We prove that $\fG\seq \fG^{\perp\perp}$. Fix $A\in \fG$ and choose
any $A'\in \fG^{\perp}$. Then (by the definition of $\fG^{\perp}$)
$A\cup A'$ is cofinite, hence $A\in \fG^{\perp\perp}$.

Remark~\ref{quot} 2) applied to $\fG\seq \fG^{\perp\perp}$ yields
$\fG^{\perp}\sep \fG^{\perp\perp\perp}$, and the reverse inclusion
follows from what we have just proved.
\end{proof}

A filter $\fG$ is said to be \emph{balanced}, if
$\fG^{\perp\perp}=\fG$. For instance, $\Cof(G)$ and $\cP(G)$ are
balanced filters. Furthermore, every balanced filter on $G$ contains
$\Cof(G)$.

\begin{fact}[\cite{DuDu1}, Thm.~14]\label{acc-bal}
Let $G$ be a linearly ordered set. Then $\CD(G)^{\perp}=\CA(G)$ and
$\CA(G)^{\perp}=\CD(G)$, therefore $\CD(G)$ and $\CA(G)$ are
balanced filters.
\end{fact}

\section{Space of $\fG$-zero functions}\label{S-Space}

Most results of this paper can be proven for normed skew fields $K$.
But to avoid technicalities, in what follows $K$ will always denote
a skew field with a discrete topology.

A left (right) $K$-linear space $L$ with a topology $\mathcal{T}$ is
said to be a \emph{linear topological space}, if the addition of
elements of $L$ defines a continuous function $L\times L\to L$,
where $L\times L$ is taken with product topology; and the same is
true for any function $k\times L\to L$ given by multiplication by $k\in K$.
Since $K$ is discrete the last condition can be
replaced by the following: for each open set $U\seq L$ and every
$0\neq k\in K$ the product $kU$ is open.

Suppose that $U_i$, $i\in I$ is a collection of subspaces of $L$
such that for all $i, j\in I$ there exists $k\in I$ with $U_k\seq
U_i\cap U_j$. A subset $V$ of $L$ is defined to be open, if for
every $a\in V$ there exists $i\in I$ such that $a+U_i\seq V$. This
defines a \emph{linear topology} $\cT$ on $L$, therefore $L$ is a
topological space with a linear topology. Note that $\cT$ is
Hausdorff iff $\cap_{i\in I} U_i=\{0\}$. In this paper we will
consider only Hausdorff linear topologies.

For instance, let $G$ be a set and let $L=\Map(G,K)$ be a left
(right) vector space of all maps from $G$ to $K$. Let $I$ be the
collection of all finite subsets of $G$, and we consider $I$ as a
set of indices. For each $i\in I$ define a subspace $U_i$ of $L$
consisting of all maps $f: G\to K$ such that $f(g)=0$ for every
$g\in i$. Then the family $U_i$, $i\in I$ defines a linear topology
on $L$ called the \emph{Tychonoff topology}. For instance, if $K$ is
a finite field (with discrete topology), then $L$ is a compact
space (Tychonoff theorem).

Suppose that $f: G\to K$ is a map. Then the \emph{support of $f$},
$\supp(f)$, is the following subset of $G$: $\supp(f) =\{g\in G\mid
f(g)\neq 0\}$. Similarly the \emph{zero-set of $f$}, $\cZ(f)$, is
defined as $\cZ(f)=\{g\in G\mid f(g)=0\}$. Clearly $G=\supp(f) \cup
\cZ(f)$ is a partition of $G$. Furthermore, if $f, h: G\to K$ and
$0\neq k\in K$ then $\cZ(f+h)\sep \cZ(f)\cap \cZ(h)$, $\cZ(k\cdot
f)=\cZ(f)$ and $\cZ(0)=G$, where $0$ stands for the zero function.

If $\fG$ is a filter on $G$, then $\FU(\fG)=\{f\in \Map(G,K)\mid
\cZ(f)\in \fG\}$ will denote the \emph{space of $\fG$-zero
functions}. Clearly $\FU(\fG)$ is a left and right subspace of the
(linear) topological space $\Map(G,K)$. For instance,
$\FU(\cP(G))=\Map(G,K)$.

The following remark shows that operations on linear spaces
$\FU(\fG)$ correspond to operations on the lattice of filters (see
Remark~\ref{oper-filt}).

\begin{lemma}\label{fil-oper}
Let $\fG_1$ and $\fG_2$ be filters on $G$. Then $\FU(\fG_1)\cap
\FU(\fG_2)=\FU(\fG_1 \cap \fG_2)$ and $\FU(\fG_1) +
\FU(\fG_2)=\FU(\fG_1 \vee \fG_2)$. Furthermore, if $\fG_1\seq \fG_2$
then $\FU(\fG_1)$ is a subspace of $\FU(\fG_2)$.
\end{lemma}
\begin{proof}
Clearly $f\in \FU(\fG_1)\cap \FU(\fG_2)$ iff $\cZ(f)\in \fG_1\cap
\fG_2$ iff $f\in \FU(\fG_1\cap \fG_2)$, which proves that
$\FU(\fG_1)\cap \FU(\fG_2)=\FU(\fG_1 \cap \fG_2)$.

To prove the inclusion $\FU(\fG_1) + \FU(\fG_2)\seq \FU(\fG_1 \vee
\fG_2)$ let $f\in \FU(\fG_1) + \FU(\fG_2)$. Then $f=h_1+h_2$, where
$\cZ(h_i)\in \fG_i$. It follows that $\cZ(f)\sep \cZ(h_1)\cap
\cZ(h_2)\in \fG_1 \vee \fG_2$, hence $f\in \FU(\fG_1 \vee \fG_2)$.

For the reverse inclusion suppose that $f\in \FU(\fG_1 \vee \fG_2)$,
that is, $\cZ(f)\sep A_1\cap A_2$ for some $A_i\in \fG_i$, $i=1, 2$. Define
$h\in \Map(G,K)$ as follows:
$$
h_1(g)=\begin{cases}
f(g),& \text{if $g\in \ov A_1$}\\
\phantom{f(g}0,& \text{if $g\in A_1$,}
\end{cases}
\hspace{15mm} h_2(g)=\begin{cases}
f(g),& \text{if $g\in  A_1\sm A_2$}\\
\phantom{fg)}0,& \text{if $g\notin A_1\sm A_2$.}
\end{cases}
$$

Then $f=h_1+h_2$ and $\cZ(h_1)\sep A_1$, $\cZ(h_2)\sep A_2$,
therefore $h_i\in \FU(\fG_i)$.
\end{proof}

Given $A\seq G$, we set $U(A,\fG)=\{f\in \FU(\fG)\mid \cZ(f)\sep
\ov A\}$. Clearly this is the same as $\supp(f)\seq A$.

The proof of the following lemma is straightforward.

\begin{lemma}\label{ua-left}
1) $U(A,\fG)$ is a left (right) subspace of $\FU(\fG)$.

2) If $A_1, A_2\seq G$ then $U(A_1\cap A_2, \fG)=U(A_1, \fG)\cap
U(A_2, \fG)$.
\end{lemma}

Now we are in a position to construct a linear topology on
$\FU(\fG)$.

\begin{theorem}[\cite{DuDu2}, Thm.~1]\label{top-UA}
1) The family of subspaces $\{U(A,\fG)\mid A\in \fG^{\pr}\}$ form a
base of zero neighborhoods of a linear topology $T(\fG)$ on the
spaces $\Map(G,K)$ and $\FU(\fG)$, and this topology is Hausdorff.

2) If $\fG$ is a  balanced filter, then $\FU(\fG)$ is complete in
this topology.

3) If $\Cof(G)\seq \fG$ then the set of all maps from $\Map(G,K)$
with finite support is dense in $\FU(\fG)$.
\end{theorem}

Let us consider some examples of topologies $T(\fG)$.

\begin{exam}\label{top}
1) If $\fG=\Cof(G)$, then $\FU(\fG)$ consists of all functions with
finite support and $\fG^\pr=\cP(G)$. Then $\pty\in \fG^\pr$, hence
$U(\pty, G)=\{0\}$ is an open set. It follows that every subset of
$\FU(\fG)$ is open and closed, hence $T(\fG)$ is a discrete
topology.

2) If $\fG=\cP(G)$, then $\fG^\pr=\Cof(G)$, $\FU(\fG)=\Map(G,K)$ and
$A$ runs over all cofinite subsets of $G$. Thus we obtain the
Tychonoff topology whose subbase is given by the subspaces
$U_g=\{f\in \Map(G,K)\mid f(g)=0\}$.

3) Suppose that $(G, \leq)$ is a linearly ordered set and
$\fG=\CA(G)$. Then $\fG^\pr=\CD(G)$, therefore the base of zero
neighborhoods is given by $U_D=\{f\in \FU(\fG)\mid f(g)=0$ for every
$g\in D\}$, where  $D$ is a well ordered subset of $G$.

Note that, if $G=(\Q, \leq)$, then the space $\Map(\Q,K)$ with
Tychonoff topology is metrizable and separable. This is not longer
true for the topology $T(\CA G)$. Indeed, suppose that $D_1, D_2,
\dots$ are well ordered subsets of $\Q$ such that $U_{D_1}, U_{D_2},
\dots$ form a basis of zero neighborhoods. This means that for every
well ordered $D\seq \Q$ there is $k$ such that $U_{D_1}\cap \ldots
\cap U_{D_k}\seq U_D$. It follows easily that $D\seq D_1\cup \ldots \cup D_k$.

Clearly there exists an ascending sequence $d_1 < d_2< \dots$ such that
$d_k\notin D_1 \cup \dots \cup D_k$ for every $k$. If $D=\{d_1, d_2, \dots\}$,
then $D$ has a d.c.c, therefore $D\seq D_1\cup \ldots \cup D_k$ for some $k$,
and then $d_k\in D_1\cup \ldots \cup D_k$, a contradiction.

This shows that in the space $(\FU(\CA \Q), T(\CA \Q))$ no point has a
countable base of neighborhoods, in particular, this space is not separable.

4) Let $G=\lan t \ran$ be an infinite cyclic group with the usual
linear ordering: $t^n \geq t^m$ iff $n\geq m$. Let $\fG=\CD(G)$,
that is, $\fG$ is generated by the following collection of sets:
$\{t^n\mid n<l\}$, $l\in \Z$. Then $\fG^\pr$ is generated by
$\{t^n\mid n > m\}$, $m\in \Z$, hence $\FU(\fG)$ is the space of
Laurent power series $\sum_{i\geq l} k_i t^i$ and $T(\fG)$ is the
$t$-adic topology.
\end{exam}

\section{Direct sum decompositions}\label{S-Dir-Sum}

Let $C$ be a subset of $G$. We identify $\Map(C,K)$ with a subspace
of $\Map(G,K)$ consisting of all maps $f: G\to K$ such that the
restriction of $f$ to $\ov C$ is zero. If $\fG$ is a filter on $G$,
then the family $\fG_C=\{A\cap C\mid A\in \fG\}$ will be a filter on
$G$ called an \emph{induced filter}.
%Clearly, if $\fG$ contains the Frechet filter on $G$ then $\fG_C$ contains the
%Frechet filter on $C$.
Note that with respect to the above identification, $\FU(\fG_C)\seq
\FU(\fG)$. Indeed, if $f\in \Map(C,K)$, then $f\in \FU(\fG_C)$ means
that $\cZ(f)=A\cap C$ for some $A\in \fG$. If we consider $f$ as a
map from $G$ to $K$, then $\cZ(f)=(A\cap C)\cup \ov C\sep A$, hence
$f\in \FU(\fG)$.

\begin{prop}\label{induced}
Let $C\seq G$. Then the topology $T(\fG_C)$ on the space
$\FU(\fG_C)$ coincides with the topology induced by $T(\fG)$.
\end{prop}
\begin{proof}
Note that $A\cap C\in \fG^\pr_C$ for every $A\in \fG^\pr$.  Indeed,
every set from $\fG_C$ can be written in the form $A'\cap C$ for
some $A'\in \fG$. Since $A\cup A'$ is cofinite in $G$, it follows
that $(A\cap C)\cup (A'\cap C)=(A\cup A')\cap C$ is cofinite in $C$.
Thus $A\cap C\in \fG_C^\pr$ and clearly $U(A,\fG)\cap \FU(\fG_C)\seq
U(A\cap C, \fG_C)$.

Now take any $A_1\in \fG_C^\pr$. We claim that $B=A_1\cup \ov C\in
\fG^\pr$ which would imply the inclusion $U(A_1, \fG_C)\seq U(B,
\fG)\cap \FU(\fG_C)$. Indeed, for every $A'\in \fG$ the set $A_1\cup
(A'\cap C)$ is cofinite in $C$. Then $(A_1\cup \ov C)\cup A'$ is
cofinite in $G$ because it contains $A_1\cup (A'\cap C)$ (cofinite
in $C$) and $\ov C$.
\end{proof}

Let $A$ be a subset of $G$. For every map $f: G\to K$ we have
$f=f\vert_A + f\vert_{\ov A}$, where $\supp(f\vert_A)\seq A$ and
$\supp(f\vert_{\ov A})\seq \ov A$. This yields a decomposition of
linear spaces: $\FU(\fG)=\FU(\fG_A)\oplus \FU(\fG_{\ov A})$, where
$\FU(\fG_A)\seq \Map(A,K)$ and $\FU(\fG_{\ov A})\seq \Map(\ov A,
K)$.

In the following proposition we will single out two important
particular cases.

\begin{prop}\label{dec}
1) If $A\in \fG$ then $\FU(\fG)=\FU(\fG_A)\oplus \Map(\ov A, K)$.
Furthermore, $T(\fG)$ induces the Tychonoff topology on $\Map(\ov A,
K)$.

2) If $\fG\sep \Cof(G)$ and $A\in \fG^\pr$, then $\FU(\fG)=U(A,
\fG)\oplus \FU(\Cof \ov A)$. Furthermore, $T(\fG)$ induces the
discrete topology on $\FU(\Cof \ov A)$.
\end{prop}
\begin{proof}
1) Since $A\cap \ov A=\pty\in \FU(\fG_{\ov A})$, the induced filter on $\ov A$ coincides
with $\cP(\ov A)$, hence $\FU(\fG_{\ov A})=\Map(\ov A, K)$. Thus, by
Proposition~\ref{induced} and Example~\ref{top} 2), the topology
induced on $\Map(\ov A, K)$ by $T(\fG)$ will be Tychonoff.
It remains to apply the equality $\FU(\fG)=\FU(\fG_A)\oplus
\FU(\fG_{\ov A})$.

2) By the definition of $\pr$-operation, we obtain $\fG_{\ov A}=\Cof(\ov A)$.
Again, by Proposition~\ref{induced} and Example~\ref{top} 1),
$T(\fG)$ induces discrete topology on $\FU(\fG_{\ov A})$. Now the
result follows from the same equality.
\end{proof}

\section{Involution and pairing}\label{S-Pairing}

Now assume that $G$ is a set with an involution $*$, that is, with a
map $G\to G$ such that $g^{**}=g$. If $A$ is a subset of $G$ then
we define  $A^*=\{a^*\mid a\in A\}$, and we put $\fG^*=\{A^*\mid
A\in \fG\}$, if $\fG$ is a filter on $G$.

The following remark is obvious.

\begin{remark}\label{star}
If $\fG$ is a filter, then $\fG^*$ is also a filter. Furthermore,
the map $\fG\to \fG^*$ defines an automorphism of the lattice of
filters on $G$. For instance, $\fG^{\pr *}=\fG^{*\pr}$ for every
filter $G$.
\end{remark}

A filter $\fG^{\pr *}$ is said to be \emph{adjoint} to the filter
$\fG$. Thus a filter $\fG$ is \emph{self-adjoint}, if $\fG^{\pr *}=
\fG$. Note that every self-adjoined filter is balanced. Indeed,
$\fG^{\pr\pr}=\fG^{\pr\pr
**}=\fG^{\pr *\pr *}=(\fG^{\pr *})^{\pr *}=\fG^{\pr *}=\fG$. For instance, let
$G$ be a linearly orderer group, $\fG=\CD(G)$ and the involution $*$
is given by $g^*=g^{-1}$. Then $\CD(G)^{\pr *}=\CA(G)^*$. Since
taking the inverse in the linearly ordered group reverses the
ordering, $\CA(G)^*=\CD(G)$, therefore $\CD(G)$ is self-adjoint, and
the same is true for $\CA(G)$.

From now on each group $G$ will be considered as a group with the
involution $*: g\to g^{-1}$.

Suppose that $G$ is a group with a linear ordering $\leq$. We say that $G$ is
\emph{left (right) ordered}, if $g_1\leq g_2$ implies $hg_1\leq hg_2$ ($g_1h\leq g_2h$)
for every $g_1, g_2, h\in G$. A group $G$ is said to be \emph{linearly ordered},
if it is left and right ordered with respect to $\leq$.

\begin{prop}\label{left-right}
A left ordered group $(G,\leq)$ is linearly ordered iff the filter
$\CD(G)$ (or $\CA(G)$) is self-adjoint.
\end{prop}
\begin{proof}
We have already proved that for a linearly ordered group $G$,
both $\CD(G)$ and $\CA(G)$ are self-adjoint filters.

Suppose that $G$ is a left ordered group and $\CD(G)$ is a
self-adjoint filter. This means that $\CD(G)^{\pr *}=\CD(G)$, which
is the same as $\CD(G)^\pr=\CD(G)^*$ or $\CA(G)=\CD(G)^*$.

To prove that $G$ is linearly ordered it suffices to check that the
cone $P=\{g\in G\mid g\geq e\}$ is invariant, that is, $a^{-1}Pa=P$
for every $a\in G$. Moreover it is enough to verify that $a^{-1} P
a\seq P$ for every $a\in G$. Assuming otherwise we will find $a,
b\in G$ such that $b > e$ and $a^{-1}ba < e$. Multiplying $ba< a$ on
the left by $b$ we obtain a descending chain $\Del=\{a> ba> b^2 a
> \dots\}$. Since $\Del$ has an a.c.c., $\ov \Del\in \CA(G)$. From $\CA(G)=\CD(G)^*$
it follows that $\ov \Del^*\in \CD(G)$ where $\Del^*=\{a^{-1}b^{-n}\mid n\in \om\}$,
hence $\Del^*$ has a smallest element. If $a^{-1}b^{-n}$ is
such, then $a^{-1} b^{-n} \leq a^{-1} b^{-n-1}$. Multiplying by $b^{n+1}a$ on the left
we obtain $b\leq e$, a contradiction.
\end{proof}

Now we define a \emph{pairing} on the space $\Map(G,K)$ as the
following partially defined non-degenerate bilinear form. If $f,
h\in \Map(G,K)$ then
\begin{equation}
\lan f, h\ran= \sum_{x\in G} f(x) h(x^*)= \sum_{x\in G} f(x^*) h(x)
\label{hhh}
\end{equation}
and the result $\lan f, h\ran$ is defined, if $\supp(f)\cap (\supp h)^*$ is a finite set.
Clearly this is the same as
$(\supp f)^*\cap \supp(h)$ is finite. In particular, this is the case
when $f\in \FU(\fG)$ and $h\in \FU(\fG^{\pr *})$ or vice versa. If
$K$ is a field, this form is symmetric.

We say that the equality (\ref{hhh}) defines a \emph{pairing} between
subspaces $L$ and $L'$ of $\Map(G,K)$, if the following holds true:

1) the product $\lan f, h \ran$ is defined for all $f\in L$, $h\in
L'$;

2) if $f\in \Map(G,K)$ and the product $\lan f, h\ran$ is defined
for every $h\in L'$, then $f\in L$;

3) if $h\in \Map(G,K)$ and the product $\lan f, h\ran$ is defined
for every $f\in L$, then $h\in L'$.

In this case $(L,L')$ is said to be a \emph{dual pair}.

\begin{prop}\label{hh}
If $\fG$ is a balanced filter on a set $G$ with an involution, then
$(\FU(\fG^{\pr*}), \FU(\fG))$ is a dual pair and the same is true
for the pair $(\FU(\fG), \FU(\fG^{\pr *}))$.
\end{prop}
\begin{proof}
1) is satisfied by the definition of $\fG^\pr$.

2) Let $f\in \Map(G,K)$ is such that $\lan f, h \ran$ is defined for
every $h\in \FU(\fG)$. Then $(\supp f)^*\cap \supp(h)$ is finite for
every $h\in \FU(\fG)$, therefore $\cZ(f)^*\cup \cZ(h)$ is a cofinite
set for every such $h$. Take any $B\in \fG$ and let $h_B$ be the
characteristic function of $G\sm B$ (that is, $h_B(g)=0$ if $g\in
B$, and $h_B(g)=1$ otherwise). Clearly $h_B\in \FU(\fG)$. It follows
that $\cZ(f)^*\cup B$ is a cofinite set, therefore, by the
definition of $\pr$, we obtain $\cZ(f)^*\in \fG^\pr$. It follows
that $\cZ(f)\in \fG^{\pr *}$ hence $f\in \FU(\fG^{\pr *})$.

3) Suppose that $h\in \Map(G,K)$ and the result $\lan f, h\ran$ is
defined for every $f\in \FU(\fG^{\pr *})$. As in the proof of 2) it
follows that $\cZ(h)\in (\fG^{\pr *})^{\pr *}=\fG^{\pr\pr}=\fG$,
since $\fG$ is balanced.
\end{proof}

Thus we obtain the following diagram of pairing:

%\marginpar{s9}
$$
%R_1 \hspace{1cm}
\vcenter{
\def\labelstyle{\displaystyle}
\xymatrix@R=20pt@C=20pt{%
\FU(\Cof
G)\ar@{^{(}->}[r]\ar@{<->}[d]&\FU(\fG_1)\ar@{^{(}->}[r]\ar@{<->}[d]&
\FU(\fG_2)\ar@{^{(}->}[r]\ar@{<->}[d]&\Map(G,K)\ar@{<->}[d]\\
\Map(G,K)&\FU(\fG_1^{\pr *})\ar@{_{(}->}[l]&\FU(\fG_2^{\pr
*})\ar@{_{(}->}[l]&
\FU(\Cof G)\ar@{_{(}->}[l]\,,\\
} }
$$
where $\longleftrightarrow$ stands for pairing, and $\fG_1\seq
\fG_2$ are balanced filters.

Because $\FU(\fG)$ and $\FU(\fG^{\pr *})$ are paired, it follows
that $\FU(\fG^{\pr *})$ is isomorphic to the space of linear
forms on $\FU(\fG)$, that is to the space of linear continuous maps from
$\FU(\fG)$ to $K$. We will derive this fact later from a more
general description of continuous linear operators on spaces of
$\fG$-zero functions.

\section{Filters on direct products}\label{S-filter}

Let $\fG$ be a filter on a set $G$, and let $\fH$ be a filter on
$H$. In this section we consider different extensions of these
filters to a filter on the direct product $H\times G$. One of these
extensions is well known.

\begin{fact}(see \cite[Sec.~6.7]{Bou})
The family $\{B\ts A\mid B\in \fH, A\in \fG\}$ is a base of a filter
$\fH\ts \fG$ on $H\ts G$. This filter is proper iff both $\fG$ and
$\fH$ are proper.
\end{fact}

But this particular filter bears no significance for topological
spaces of formal sums. The following filter is more useful.

\begin{fact}[see \cite{DuDu1}, Prop.~16]\label{ten}
The family $\{\ov{\ov B\ts \ov A}\mid B\in \fH, A\in \fG\}$ form a
base of a filter $\fH\ots \fG$ on $H\ts G$. The filter $\fH\ots \fG$
is proper iff $\fG$ or $\fH$ is proper. Furthermore, $\fH\ots
\fG\seq \fH\ts \fG$.
\end{fact}

Note that, if $X$ is a subset of $H\ts G$, then $X\in \fH\ots \fG$
iff $X\sep \ov{\ov B\ts \ov A}$ iff $\ov X\seq \ov B\ts \ov A$ for
some $B\in \fH$ and $A\in \fG$.

\begin{fact}[see \cite{DuDu1}, (6)]\label{coff}
$\Cof(H)\ots \Cof(G)= \Cof(H\ts G)$.
\end{fact}

We introduce a new filter $\Cof(\fH, \fG)$ whose subbase is given by
the following collections of sets: $\{\ov{ h \ts \ov A}\mid h\in
H, A\in \fG\}$ and $\{\ov{\ov B\ts g} \mid B\in \fH, g\in G\}$. Thus
$\Cof(\fH, \fG)$ consists of subsets of $H\ts G$ whose complement is
a subset of a finite union of sets $\{h\}\ts \ov A$ and $\ov B\ts
\{g\}$. But every subset of $\{h\}\ts \ov A$, $A\in \fG$ is of the
form  $\{h\}\ts \ov A'$ for some $A'\in \fG$, and every subset of
$\ov B\ts \{g\}$, $B\in \fH$ is of the form $\ov B'\ts \{g\}$ for
some $B'\in \fH$. It follows that $\Cof(\fH, \fG)$ consists of
complements to the sets $\lt(\bigcup_{i=1}^n \{h_i\}\ts \ov A_i\rt)
\cup \lt(\bigcup_{j=1}^m \ov B_j\ts \{g_j\}\rt)$, where $g_j\in G$,
$h_i\in H$ and $A_i\in \fG$, $B_j\in \fH$.

\begin{fact}[see \cite{DuDu1}, Prop.~18]\label{cof}
$\Cof(\fH, \fG)=(\Cof H\ots \fG)\vee (H\ots \Cof G)$.
\end{fact}

\begin{fact}[see \cite{DuDu1}, Prop.~18, 20]\label{hg}
Suppose that $\fG$ contains the Frechet filter on $G$, and $\fH$
contains the Frechet filter on $H$. Then

1) $\Cof(H\ts G)\seq \Cof(\fH, \fG)\seq \fH\ots \fG$.

2) $\Cof(\fH, \fG)=(\Cof H\ts \Cof G)\cap (\fH\ots\fG)$.
\end{fact}

Now we are in a position to introduce the main construction of this
section.

\begin{definition}\label{defn:dek}
Suppose that $\fH$ is a filter on $H$ and $\fG$ is a filter on $G$.
Let $\lan \fH, \fG\ran$ consists of all subsets $X$ of $H\times G$
with the following properties:

a) for every $A\in \fG^\pr$ there exists $B'\in \fH$ such that
$B'\ts \ov A\seq X$;

b) for every $B\in \fH^\pr$ there exists $A'\in \fG$ such that $\ov
B\ts A'\seq X$.
\end{definition}

If $t\in G$, then define $H_t(X)=\{h\in H\mid (h,t)\in X\}$. Similarly,
if $s\in H$, then set $G_s(X)=\{g\in G\mid (s,g)\in X\}$.

It is easily seen that $X\in \lan \fH, \fG\ran$ iff the following
holds:

a)$'$ \ $\bigcap_{t\in \ov A} H_t(X)\in \fH$ for every $A\in \fG^\pr$;

b)$'$ \ $\bigcap_{s\in \ov B} G_s(X)\in \fG$ for every $B\in \fH^\pr$.

For instance, the equivalence of a) and a)$'$ can be seen as
follows. If a) holds then $B'\seq \bigcap_{t\in \ov A} H_t$, hence
this intersection is in $\fH$; and if a)$'$ holds then we take $B'$
to be equal to this intersection.

\begin{fact}[see \cite{DuDu1}, Thm.~23]\label{main-fil}
$\lan \fH, \fG\ran$ is a filter on $H\ts G$ containing $\fH\ots
\fG$.
\end{fact}

\begin{fact}[see \cite{DuDu1}, L.~24]\label{inter}
If $\fH$ and $\fG$ contain Frechet filters, then $(\fH\ots \fG)\cap
(\fH^\pr \ots \fG^\pr)=\Cof(H\ts G)$.
\end{fact}

The following remark is obvious.

\begin{remark}\label{incl}
If $\fH\seq \fH'$ and $\fG\seq \fG'$, then $\fH\ots \fG\seq \fH'\ots
\fG'$.
\end{remark}

\begin{prop}\label{ttt}
If $\fH$ and $\fG$ are balanced filters, then $\lan \fH, \fG\ran\sep
\Cof(\fH,\fG):(\fH^\pr\ots\fG^\pr)=(\fH^\pr\ots\fG^\pr)^\pr$.
\end{prop}
\begin{proof}
First we prove the equality. We claim that $\Cof(\fH, \fG)\cap
(\fH^\pr\ots\fG^\pr)= \Cof(H\ts G)$. Indeed, by Fact~\ref{hg} 2),
we obtain
$$\Cof(\fH, \fG)\cap (\fH^\pr\ots\fG^\pr)=(\Cof H\ts \Cof G)\cap
(\fH\ots \fG)\cap (\fH^\pr\ots\fG^\pr)\,.$$
By Fact~\ref{inter} this is the same as $(\Cof H \ts \Cof G)\cap \Cof(H\ts G)=\Cof(H\ts G)$, as
desired.

Then by Remark~\ref{quot} 3) and the definition of $\pr$ we obtain
$\Cof(\fH, \fG):(\fH^\pr\ots \fG^\pr)=(\Cof(\fH,\fG)\cap
(\fH^\pr\ots \fG^\pr)): (\fH^\pr\ots \fG^\pr)=\Cof(H\ts G):
(\fH^\pr\ots \fG^\pr)=(\fH^\pr\ots \fG^\pr)^\pr$.

Now we prove the inclusion. Suppose that $Z\in (\fH^\pr\ots
\fG^\pr)^\pr$ and we have to show that $Z\in \lan \fH, \fG\ran$.
Take any $B\in \fH^\pr$ and define $A'\seq G$ by the following rule:
$\ov A'=\{g\in G\mid (h,g)\in \ov Z$ for some $h\in \ov
B\}=\pi_G[(\ov B\ts G)\cap \ov Z]$, where $\pi_G$ is a projection on $G$.

Suppose that $A'\notin \fG$. Then $A'\notin (\fG^\pr)^\pr$ because
$\fG$ is balanced. It follows that there exists $A\in \fG^\pr$ such
that the union $A'\cup A$ is not cofinite in $G$, that is, the
intersection $\ov A'\cap \ov A$ is infinite. Since $\ov{\ov B\ts \ov
A}\in \fH^\pr \ots \fG^\pr$, it follows that $Z\cup \ov{\ov B\ts \ov
A}$ is cofinite in $H\ts G$, hence $\ov Z\cap (\ov B\ts \ov A)$ is a
finite set. On the other hand for every $g\in \ov A'\cap \ov A$
(there are infinitely many of them) there exists $h\in \ov B$ such
that $(h,g)\in \ov Z\cap (\ov B\ts \ov A)$, hence this set must be
infinite, a contradiction. Thus $A'\in \fG$.

We prove that $\ov B\ts A'\seq Z$. Indeed, otherwise $(h,g)\in \ov
Z$ for some $(h,g)\in \ov B\ts A'$. Since $g\notin \ov A'$, by the
construction of $\ov A'$ we obtain $h\notin \ov B$, a contradiction.

Thus for every $B\in \fH^\pr$ there exists $A'\in \fG$ such that
$\ov B\ts A'\seq Z$. Similarly for every $A\in \fG^\pr$ there exists
$B'\in \fH$ such that $B'\ts \ov A\seq Z$. It follows that $Z\in
\lan \fH, \fG\ran$.
\end{proof}

Before proving the next lemma, let us recall a useful equality: if
$B\seq H$ and $A\seq G$, then $\ov{(\ov B\ts \ov A)}= (B\ts A)\cup
(B\ts \ov A)\cup (\ov B\ts A)$.

\begin{lemma}\label{ss}
$\fH^\pr \ots \fG^\pr\seq \lan \fH, \fG\ran^\pr$ for all filters
$\fH$ and $\fG$.
\end{lemma}
\begin{proof}
Suppose that $Z\in \fH^\pr \ots \fG^\pr$, hence $\ov Z\seq \ov B\ts
\ov A$ for some $B\in \fH^\pr$, $A\in \fG^\pr$. We have to prove
that $Z\in \lan \fH, \fG\ran^\pr$, that is, $Z\cup X$ is a cofinite
set for every $X\in \lan \fH, \fG\ran$. Since $X\in \lan \fH,
\fG\ran$ there are $B'\in \fH$ and $A'\in \fG$ such that $B'\times
\ov A, \ov B\ts A'\seq X$. Then
$$
\ov X\seq \ov{(B'\ts \ov A) \cup
(\ov B\ts A')}=\ov{B'\ts \ov A} \cap \ov{\ov B\ts A'}=
$$
$$
((\ov B'\ts A)\cup (\ov B'\ts \ov A)\cup (B'\ts A))\cap ((B\ts \ov A')\cup (B\ts
A')\cup (\ov B\ts \ov A'))\,.
$$
 Intersecting this with $\ov Z\seq \ov
B\ts \ov A$ and taking into account that the intersections of $\ov
B\ts \ov A$ with $\ov B'\ts A$, $B'\ts A$, $B\ts \ov A'$ and $B\ts
A'$ are empty, we obtain $\ov X\cap \ov Z\seq (\ov B'\cap \ov B)\ts
(\ov A\cap \ov A')$. Since $B\in \fH^\pr$ and $B'\in \fH$, therefore
$B\cup B'$ is a cofinite set, hence $\ov B'\cap \ov B$ is finite.
Similarly $\ov A'\cap \ov A$ is a finite set, hence $\ov X\cap \ov
Z$ is finite, as desired.
\end{proof}

\begin{remark}\label{ua}
$\fH\ots\fG \seq \lan \fH^\pr, \fG^\pr \ran^\pr$ for all filter
$\fH$ and $\fG.$
\end{remark}
\begin{proof}
Indeed, by Lemma~\ref{two-perp} we have $\fG\seq \fG^{\pr\pr}$ and
$\fH\seq \fH^{\pr\pr}$, therefore $\fH\ots \fG\seq \fH^{\pr\pr}\ots
\fG^{\pr\pr}$ by Remark~\ref{incl}. Furthermore
$\fH^{\pr\pr}\ots\fG^{\pr\pr} \seq \lan \fH^\pr, \fG^\pr \ran^\pr$
by Lemma~\ref{ss}.
\end{proof}

\begin{theorem}\label{ht}
Suppose that $\fH$ and $\fG$ are balanced filters. Then $\lan \fH,
\fG\ran =(\fH^\pr\ots \fG^\pr)^\pr$. In
particular, $\lan \fH, \fG\ran$ is a balanced filter.
\end{theorem}
\begin{proof}
By Proposition~\ref{ttt} we have $\lan \fH, \fG\ran \sep \lan
\fH^\pr\ots \fG^\pr\ran^\pr$. On the other hand applying $\pr$ to
the inclusion in Lemma~\ref{ss}, we obtain $(\fH^\pr\ots
\fG^\pr)^\pr\sep \lan \fH, \fG\ran^{\pr\pr}$. Thus $\lan \fH,
\fG\ran \sep \lan \fH^\pr, \fG^\pr\ran^\pr\sep \lan \fH,
\fG\ran^{\pr\pr}$, hence $\lan \fH, \fG\ran\sep \lan \fH,
\fG\ran^{\pr\pr}$. Then Lemma~\ref{two-perp} yields the desired.
\end{proof}

\section{$\fG$-sums}\label{S-sums}

Suppose that $(X,\Ta)$ is a typological abelian group and $x_i$,
$i\in I$ is a family of elements of $X$. An element $x\in X$ is said to be
a \emph{sum} of this family with respect to $\Ta$, written
$x=\sum_{i\in I}^\Ta x_i$, if the following holds. For every
neighbourhood $U$ of $x$ there is a finite subset $\Del\seq I$ such
that $\sum_{i\in\Delta'}x_i\in U$ for every finite set $\Del'\seq I$
containing $\Del$. Clearly, if $X$ is Hausdorff, then the sum is
unique. It is easily seen that, if the family $x_i$, $i\in I$ is
summable, then the limit of the $x_i$ with respect to the Frechet
filter on $I$ is equal to zero. This means that for every zero
neighbourhood $U$  there exists a finite subset $\Del$ of $I$ such
that $x_i\in U$ for every $i\in I\sm \Del$.

Suppose that $(X,\Ta_X)$ and $(Y,\Ta_Y)$ are topological abelian
groups. If $\phi$ is a continuous morphism from $X$ to $Y$ then
$\phi$ preserves topological sums. This means that, if $x=\sum_{i\in
I}^{\Ta_X} x_i$ in $X$, then the sum $\sum_{i\in I}^{\Ta_Y}
\phi(x_i)$ exists and equal to $\phi(x)$.

For more on sums in topological abelian groups see \cite{BuTop}

Now we define a sum with respect to a filter.

\begin{definition}\label{deff}
Let $h_j\in \Map(G,K)$, $j\in J$ be a family of maps and let $\fG$
be a filter. This family is said to be \emph{$\fG$-summable}, and
the map $f: G\to K$ is a \emph{$\fG$-sum} of this family,
$f=\sum_{j\in J}^{\fG} h_j$, if the following holds:

1) for every $g\in G$ there are only finitely many $j\in J$ such
that $g\in \supp(h_j)$, and $f(g)=\sum_{j\in J} h_j(g)$;

2) $\bigcap_{j\in J} \cZ(h_j)\in \fG$.
\end{definition}

Note that the condition 1) of this definition means that the family
$\{h_j\}$, $j\in J$ is summable with respect to Tychonoff topology
on $\Map(G,K)$. Furthermore, 2) implies that $h_j\in \FU(\fG)$ for
every $j$; and $\cZ(f)\sep \bigcap_{j\in J} \cZ(h_j)$ yields that
$f\in \FU(\fG)$.

In the following theorem we compare these two types of summability.

\begin{theorem}\label{sum}
1) If the family $h_j$, $j\in J$ is $\fG$-summable, then it is
summable with respect to topology $T(\fG)$. Furthermore, $\sum_{j\in
J}^{\fG} h_j= \sum_{j\in J}^{T(\fG)} h_j$.

2) If the family $h_j\in \FU(\fG)$, $j\in J$ is summable with
respect to $T(\fG)$ and $\fG$ is balanced, then it is $\fG$-summable
and $\sum_{j\in J}^{\fG} h_j= \sum_{j\in J}^{T(\fG)} h_j$ again.
\end{theorem}
\begin{proof}
1) Suppose that $f=\sum_{j\in J}^{\fG} h_j$ and $A=\bigcap_{j\in J}
\cZ(h_j)\in \fG$. Choose any $A'\in \fG^\pr$. Then $A\cup A'$ is a
cofinite set, hence $A\cup A' = G\sm \{g_1, \dots, g_n\}$, $g_i\in
G$. Let $J_0$ consist of all $j\in J$ such that $g_t\in \supp(h_j)$
for some $t=1, \dots, n$. By the assumption $J_0$ is a finite subset
of $J$.

Suppose that $J'$ is any finite subset of $J$ containing $J_0$. Then
$\cZ(f-\sum_{j\in J'} h_j)\sep A\cup \{g_1, \dots, g_n\}\sep \ov
A'$, therefore $f-\sum_{j\in J'} h_j\in U(A', \fG)$. This proves
that $f=\sum_{j\in J}^{T(\fG)} h_j$.

2) Because the family $h_j$, $j\in J$ is summable with respect to
$T(\fG)$, by what we have already noticed, $h_j$ converges to $0$
with respect to the Frechet filter on $I$. It follows that for every
$g\in G$ there exist only finitely many $j\in J$ such that
$h_j(g)\neq 0$, therefore 1) holds true.

Thus it remains to prove that $A=\bigcap_{j\in J} \cZ(h_j)\in \fG$.
Let $f=\sum_{j\in J}^{T(\fG)} h_j$. As above the $h_j$ converge to
zero with respect to the Frechet filter on $J$. Thus for every
$A'\in \fG^\pr$ there is a finite set of indices $F(A')$ such that
for every $j\in J\sm F(A')$ we have $h_j\in U(A', \fG)$. Then
$\cZ(h_j)\sep \ov A'$ yields $\bigcap_{j\in J\sm F(A')} \cZ(h_j)\sep
\ov A'$. Since $h_j\in \FU(\fG)$ and $F(A')$ is finite, it follows
that $A_1=\bigcap_{j\in F(A')} \cZ(h_j)\in \fG$. Then
$$A\cup A'=
\bigl(\bigcap_{j\in J\sm F(A')} \cZ(h_j) \cup A'\bigl)\cap
\bigl(\bigcap_{j\in F(A')} \cZ(h_j) \cup A'\bigl)=G\cap (A_1\cup
A')=A_1\cup A'
$$
is a cofinite set. Since this is true for any $A'\in
\fG^\pr$, we conclude that $A\in \fG^{\pr\pr}=\fG$.
\end{proof}

\section{Matrix notations}\label{S-matrix}

Suppose that $G$ is a set and $H$ is a set with an involution $*$.
Each map $\Psi: H\times G\to K$ can be consider as an $H\times G$
matrix over $K$ whose $(h,g)$-entry, $\Psi^g_h$, is $\Psi(h,g)$.
These notations resemble the notations in tensor calculus and, as we
will see below, they are quite advantageous, when we consider
multiplication of matrices.

The set of all such maps form a (left and right) vector space over
$K$ and will be denoted by ${}^H K^G$. If $H$ consists of one
element, a map $\Psi: H\times G\to K$ is said to be a \emph{row},
and we use small Greek letters $\al, \beta, \dots$ to denote rows.
Similarly, if $G$ consists of one element, then a map $\Psi: H\to K$
is a said to be a \emph{column}, and we use small boldfaced letters
$\mathbf{a}, \mathbf{b}$, $\dots$ to denote columns. In case, when
$H$ consists of one element $h$, we simplify notations:
$K^G={}^{\{h\}} K^G$; and similarly ${}^H K$ means ${}^H K^{\{g\}}$,
when $G$ consists of one element $g$.

Let $\del$ be the Kronecker symbol on $G$, that is, $\del$ is a map
from $G\times G$ to $K$ such that $\del^g_h=1$ if $g=h$ and
$\del^g_h=0$ otherwise. If $\Psi\in {}^HK^G$ and $h\in H$, then
$\Psi_h$ will denote the row of $\Psi$ with number $h$, therefore
$(\Psi_h)^g=\Psi^g_h$ for every $g\in G$. Similarly, $\Psi^g$ will
denote the column of $\Psi$ with number $g$, therefore
$(\Psi^g)_h=\Psi^g_h$ for every $h\in H$. In particular, $\del_g$ is
a row whose $g$th entry is $1$ and all remaining entries are zero,
and similarly for the column $\del^g$.

Suppose that $\Phi\in {}^JK^H$ and $\Psi\in {}^HK^G$. We say the the
product $\Theta=\Phi \cdot \Psi\in {}^J K^G$ is defined if, for
every pair $(j,g)\in J\times G$, we have $\Phi^{h^*}_j\cdot
\Psi^g_h\neq 0$ only for finitely many $h\in H$ and
$\Theta^g_j=\sum_{h\in H} \Phi^{h^*}_j \cdot \Psi^g_h$. This defines
a partial operation $^J K^H \times {}^HK^G\to {}^J K^G$.

Note that $\sum_{h\in H} \Phi^{h^*}_j \cdot \Psi^g_h=\sum_{h\in H}
\Phi^h_j \cdot \Psi^g_{h*}$. More precisely, the left and right
parts are defined simultaneously and, if they are defined, they are
equal. Immediately from the definition it follows that $\Th_j=
\Phi_j \cdot \Psi$ and $\Th^g= \Phi\cdot \Psi^g$.

If $G$ is a finite set, then the multiplication on $^GK^G$ is
defined everywhere, therefore $^GK^G$ is a ring isomorphic to the
ring of $|G|\times |G|$ matrices over $K$. But the unity of this
ring is given by the map $E: G\times G\to G$ such that $E^{g^*}_g=1$
for every $g\in G$ and zero otherwise. For instance, if $G$ consists
of one element $g$, then the ring $^{\{g\}}K^{\{g\}}$ is isomorphic
to $K$. If $G$ is infinite, the partial multiplication we have just
defined is usually not associative. Indeed, suppose that $G=\N$ with
the identical involution $*$, and let $\Phi$, $\Psi$ and $\Th$ be
the following matrices:

$$\Phi=\begin{pmatrix}
1&1&1&\dots\\
\end{pmatrix},
\quad \Psi=\begin{pmatrix}
1&-1&0&0&\dots\\
0&1&-1&0&\dots\\
0&0&1&-1&\dots\\
\vdots&\vdots&\vdots&\vdots&\ddots\\
\end{pmatrix}
\mbox{\ and \ } \Th=\begin{pmatrix}
1\\
1\\
1\\
\vdots\\
\end{pmatrix}.
$$

Then $(\Phi\cdot \Psi)\cdot \Th=1$ and $\Phi\cdot (\Psi\cdot
\Th)=0$.

The following lemma claims distributivity and can be easily verified
by direct calculations.

\begin{prop}\label{dist}
Suppose that $\Phi\in {}^JK^H$ and $\Psi, \Th\in {}^HK^G$. If both
products $\Phi\cdot \Psi$ and $\Phi\cdot \Th$ are defined, then
$\Phi\cdot (\Psi+\Th)$ is also defined and equal to $\Phi\cdot \Psi
+ \Phi\cdot \Th$.
\end{prop}

An arbitrary row $\ga\in K^G$ is uniquely determined by its
coordinates $\ga^g\in K$, $g\in G$. Thus $\ga$ is a topological sum
(see Section~\ref{S-sums}) with respect to Tychonoff topology on
$\Map(G,K)$: $\ga= \sum_{g\in G}^{\Tych} \ga^g \del_g$. Furthermore,
if $\ga\in \FU(\fG)$, then clearly $\ga= \sum_{g\in G}^{\fG} \ga^g
\del_g$. Similarly,  each column $\ba\in {}^H K$ is uniquely
determined by its coordinates $a_h\in K$, $h\in H$, hence we can
write $\ba=\sum_{h\in H}^{\Tych} \del^h a_h$; and if $\ba\in
\FU(\fH)$, then $\ba:=\sum_{h\in H}^{\fH} \del^h a_h$.

\section{Continuous linear maps}

\def\Fresh{Frechet}

Recall that $\fG$ is a filter on a set $G$ and $\fH$ is a filter on
a set $H$. In what follows we will always assume that $\fG$ and
$\fH$ contain Frechet filters. Thus $\FU(\Cof G)\seq \FU(\fG)$ and
$\FU(\Cof H)\seq \FU(\fH)$. In this section we describe continuous
linear maps of topological linear spaces $\FU(\fG)\to \FU(\fH)$.

We will consider $\FU(\fG)$ as a left or right $K$-vector space. To
specify the side, we will use $K\{\fG\}$ to denote this space
considered as a left vector space over $K$, and call it a
\emph{space of $\fG$-zero rows}. Similarly $\{\fG\}K$ will denote
$\FU(\fG)$ considered as a right vector space over $K$, and will be
called a \emph{space of $\fG$-zero columns}. We use $K[G]$ to denote
$K\{\Cof G\}$, and $[G]K$  to denote $\{\Cof G\}K$. If $G$ is a
group, then $K[G]$ with the operations defined in
Section~\ref{S-matrix} is isomorphic to the usual group ring. In
what follows we will always assume that $G$ and $H$ are endowed with
an involution.

Recall that we agreed to denote rows with small Greek letters $\al,
\beta \dots$, and columns with small boldfaced letters $\ba, \bb,
\dots$. Suppose that $\phi$ is a map from $\FU(\fG)$ to $\Map(H,K)$.
First we consider $\FU(\fG)$ as a right vector space $\{\fG\} K$ and
$\Map(H,K)$ as the space of columns ${}^HK$. The image of a column
$\ba\in \{\fG \} K$ will be denoted by $\ph[\ba]$. Then we can
assign to $\phi$ an $H\times G$ matrix $\Phi$, whose $(h,g^*)$-entry
is equal to $\phi[\del^g]_h:$
$$\Phi^{g^*}_h=\ph[\delta^g]_h\,.$$
(Recall that $\del^g$ is a column whose $g$th coordinate is $1$ and
all the remaining coordinates are zero, and $\ph[\delta^g]_h$ is the
$h$th coordinate of the column $\ph[\delta^g]$). We say that $\Phi$
is a \emph{matrix of $\phi$}. For instance,
$\Phi^{g^*}=\phi[\del^g]$ is the $g^*$th column of $\Phi$.

The \emph{zero set of $\Phi$}, $\cZ(\Phi)$, is a collection of all
$(h,g)\in H\ts G$ such that $\Phi^g_h=0$. This is in accordance
with our previous definition of the zero set of a map.

%On the other hand we may consider $\FU(\fG)$ as a left $K$-vector space $K\{\fG\}$
%and $\Map(H,K)$ as a space of rows $K^H$. In this case the image of a row
%$\al\in K\{\fG\}$ will be denoted by $[\al]\phi$. Then we can assign to $\phi$
%an $G\times H$ matrix $\Phi^t$ whose $(g^*,h)$-entry is equal to
%$([\del_g] \phi)^h:$
%$$(\Phi^t)^h_{g^*}=([\del_g] \phi)^h\,.$$
%(Recall that $\del_g$ is a row whose $g$th is $1$ and all the remaining entries are
%$0$, and $([\del_g] \phi)^h$ is the $h$th coordinate of the column $([\del_g] \phi)^h$).
%Clearly $\Phi^t$ is a transpose of $\Phi$.

Now we restrict ourselves to the case, when $\phi$ is a map from
$\{\fG\}K$ to $\{\fH\}K\seq \Map(H,K)$.

\begin{definition}\label{defn:Glin}
A right (left) linear map $\ph:\FU(\fG)\to \FU(\fH)$ is said to be
\emph{$\fG$-linear}, if for any $\fG$-summable family $h_j$, $j\in
J$, where $h_j\in \FU(\fG)$, the family $\ph[h_j]$, $j\in J$ is
$\fH$-summable and
$$\ph\bigl[\suml_{j\in J}{}^{\fG}{h}_j\bigl]=\suml_{j\in J}{}^{\fH}\ph[{h}_j].$$
\end{definition}

Apparently this condition bears no connection with topology. But below
(see Theorem~\ref{thm:endo}) we will see that $\fG$-linearity is the same
as continuity. The following lemma shows that both $\fG$-linear or continuous
linear operators are uniquely determined by their matrices.

\begin{lemma}\label{lin}
Suppose that $\phi: \{\fG\}K\to \{\fH\}K$ is either $\fG$-linear or a
linear continuous map. If $\Phi$ is a matrix of $\phi$ and $\ba\in \{\fG\}K$,
then $\phi[\ba]=\Phi\cdot \ba$.
\end{lemma}
\begin{proof}
We will prove this lemma only when $\phi$ is continuous. The proof
in the case, when $\phi$ is $\fG$-linear, is similar.

As we have already noticed (see a remark after Lemma~\ref{dist}),
$\ba=\suml_{g\in G}{}^{\fG} \del^g a_g$, $a_g\in K$, hence
$\ba=\suml_{g\in G}{}^{T(\fG)} \del^g a_g$ by Theorem~\ref{sum}.
Since $\phi$ is continuous and linear, we obtain
$\phi[\ba]=\suml_{g\in G}{}^{T(\fH)} \phi[\del^g] a_g$. Invoking
Theorem~\ref{sum} again, we get $\phi[\ba]=\suml_{g\in G}{}^{\fH}
\phi[\del^g] a_g=\suml _{g\in G}{}^{\fH} \Phi^{g^*} a_g$. On the
other hand, by the definition of product of matrices, $(\Phi\cdot
\ba)_h = \suml_{g\in G} \Phi^{g^*}_h a_g$, therefore
$\phi[\ba]=\Phi\cdot \ba$, as desired.
\end{proof}

Our next objective is to decide when the left multiplication by an
$H\times G$ matrix $\Phi$ defines a continuous linear map $\phi$
from $\{\fG\} K$ to  $\{\fH\} K$. As a first approximation we
consider the following condition:

\begin{equation}\label{eq:m1}
\bigcap_{s\in \ov B}\Za(\Phi_s)\in\fG^{\bot*} \ \text{ for every } \
B\in\fH^\bot\,.
\end{equation}

It guarantees that $\phi$ is a linear map.

\begin{remark}\label{prop:ex}
Let $\Phi$  be an $H\times G$-matrix satisfying (\ref{eq:m1}). Then

a) The rows of $\Phi$ belong to the space $K\{\fG^{\bot*}\}$.

%In particular, if $\Phi$ is a row (that is, $H$ consists of one element), then
%(\ref{eq:m1}) is equivalent to $\Phi$ being in $K\{\fG^{\bot*}\}$;

b) the rule $\ba\to\Phi\cdot\ba$, where $\ba\in\{\fG\}K$, defines a
$K$-linear map $\phi: \{\fG\}K\rightarrow {}^HK$.
\end{remark}
\begin{proof}
a) If $B=H\sm \{h\}$, then $\ov B=\{h\}$. Since $B\in \fH$,
(\ref{eq:m1}) yields $\cZ(\Phi_h)\in \fG^{\pr *}$.

b) By what we have just proved, $\cZ(\Phi_h)\in \fG^{\pr *}$ for
every row $\Phi_h$ of $\Phi$. Since $\ba\in \{\fG\} K$ and
$\cZ(\Phi_h)\in \fG^{\pr *}$, the product $\Phi_h\cdot \ba$ is
defined (see Proposition~\ref{hh}) and belongs to $K$. Then
$\Phi\cdot \ba$ is defined and belongs to ${}^HK$.
\end{proof}

We need to put one extra condition on $\Phi$ to ensure that the
image of $\phi$ is contained in $\{\fH\} K$.

\begin{equation}\label{eq:m2}
\bigcap_{t\in \ov{A}}\Za(\Phi^{t^*})\in\fH \ \text{ for every } \
A\in\fG.
\end{equation}

%Remark that if $H$ is a singleton then this condition is fulfilled.

The next proposition shows  that (\ref{eq:m1}) and  (\ref{eq:m2})
together imply that $\phi$ is continuous.

\begin{prop}\label{prop:cont}
Let $\Phi$ be an $H\times G$-matrix satisfying (\ref{eq:m1}) and
(\ref{eq:m2}). Then the rule $\ba\to\Phi\cdot\ba$ defines a
continuous linear map $\phi: \{\fG\}K\to\{\fH\}K$.
\end{prop}
\begin{proof}
By Remark~\ref{prop:ex}, $\Phi\cdot\ba\in {}^H K$. We prove that
$\Phi\cdot \ba\in \{\fH\} K$. Indeed, from $\ba\in \{\fG\}K$ it
follows that $A=\Za(\ba)\in\fG$. Furthermore, (\ref{eq:m2}) implies
that $B=\bigcap_{t\in \ov{A}}\Za(\Phi^{t^*})\in \fH$, therefore
$B\times \ov{A^*}\subseteq\Za(\Phi)$. To show that $\Phi\cdot\ba\in
\{\fH\}K$, take any $h\in B$. Then

$$(\Phi\cdot \ba)_h=\suml_{t\in G}\Phi_h^{t^*} a_t=
\suml_{t\in A}\Phi_h^{t^*}\cdot0+\suml_{t\in \ov A}0\cdot a_t=0.
$$

It follows that $\Za(\Phi\cdot\ba)\supseteq B$, therefore
$\Za(\Phi\cdot\ba)\in\fH$, as desired.

As we have already noticed (see Lemma~\ref{dist}),
$\Phi\cdot(\ba+\bb)=\Phi\cdot \ba+\Phi\cdot \bb$ for all columns
$\ba, \bb\in \{\fG\} K$. Furthermore, clearly $(\Phi\cdot \ba)\cdot
k=\Phi\cdot(\ba \cdot k)$ for every $k\in K$. It follows easily that
$\phi$ is linear.

To prove that $\phi$ is continuous, it suffices to check that $\phi$
is continuous at zero. This means that for every zero neighbourhood
$U(B,\fH)$, $B\in\fH^\bot$ in $\{\fH\}K$, there exists a zero
neighbourhood $U(A,\fG)$, $A\in\fG^\bot$ in $\{\fG\}K$ such that
$$\Phi\cdot U(A,\fG)\subseteq U(B,\fH).$$

From (\ref{eq:m1}) we obtain $C=\bigcap_{s\in \ov B}\Za(\Phi_s)\in
\fG^{\bot*}$. If $A=C^*$, then $A\in\fG^\bot$, and we prove that $A$
is as required. By the definition of $A$, we have $\ov B\times
A^*\subseteq\Za(\Phi)$, that is, $\Phi_s^{g^*}=0$ for any $s\in\ov
B$ and $g\in A$. Take any $\ba\in U(A,\fG)$, hence $a_g=0$ for every
$g\in\ov A$.  Then for any $s\in \ov B$ we obtain:
\[(\Phi\cdot \ba)_s=\suml_{g\in G}\Phi^{g^*}_s a_g=
\suml_{g\in \ov A}\Phi^{g^*}_s\cdot0+\suml_{g\in A}0\cdot a_g=0.\]
It follows that $\Za(\Phi\cdot \ba)\supseteq \ov B$, therefore
$\Phi\cdot \ba\in U(B,\fH)$.
\end{proof}

Below (see Theorem~\ref{thm:endo}) we will show the the converse is
also true: if $\ph:\{\fG\}K\to\{\fH\}K$ is a continuous linear map,
then its matrix $\Phi$ satisfies both (\ref{eq:m1}) and
(\ref{eq:m2}). But first we connect these conditions with filters on
direct products.

\begin{lemma}\label{prop:MM}
A matrix $\Phi\in{}^HK^G$ satisfies (\ref{eq:m1}) and (\ref{eq:m2})
iff $\Za(\Phi)\in\lan\fH,\fG^{\bot*}\ran$.
\end{lemma}

\begin{proof}
Let us rewrite a)$'$ from  Definition~\ref{defn:dek} replacing
$\fG$  by $\fG^{\bot*}$ and $X$ by $\Za(\Phi)$. Then $\fG^\bot$
should be replaced by $\fG^*$.

\begin{equation}\label{eq:MM2}
\Za(\Phi)\in\langle \fH,\fG^{\bot*}\rangle \LR
\begin{cases}
\ \bigcap\limits_{t\in \ov A} H_t\in\fH \ \text{ for every } \ A\in\fG^*;\\
\ \bigcap\limits_{s\in\ov B} G_s\in \fG^{\bot*} \ \text{ for every } \ B\in\fH^\bot\,,\\
\end{cases}
\end{equation}
where $H_t=\set{h\in H\mid (h,t)\in \Za(\Phi)}$ and $G_s=\set{g\in
G\mid (s,g)\in \Za(\Phi)}$.

But clearly $H_t=\Za(\Phi^t)$ and $G_s=\Za(\Phi_s)$. Then
(\ref{eq:MM2}) can be rewritten as follows:

\begin{equation}\label{eq:MM3}
\Za(\Phi)\in\langle \fH,\fG^{\bot*}\rangle\LR
\begin{cases}
\ \bigcap\limits_{t\in \ov A}\Za(\Phi^t)\in\fH \ \text{ for every } \ A\in\fG^*;\\
\ \bigcap\limits_{s\in\ov B}\Za(\Phi_s)\in\fG^{\bot*} \ \text{ for
every } \ B\in\fH^\bot.
\end{cases}
\end{equation}
Applying the involution, we see that the first condition in
(\ref{eq:MM3}) is equivalent to (\ref{eq:m2}). Furthermore, the
second condition in (\ref{eq:MM3}) coincides with (\ref{eq:m1}).
\end{proof}

The following theorem characterizes continuous linear maps between
spaces $\{\fG\}K$ and $\{\fH\}K$ in terms of filters on $H\times G$.

\begin{theorem}\label{thm:endo}
Suppose that $\fH$ and $\fG$ are balanced filters on $H$ and $G$ and
$\Phi$ is an $H\ts G$-matrix. Then the following are equivalent.

a) The left multiplication by $\Phi$ defines a continuous linear map
$\phi$ from $\{\fG\}K$ to $\{\fH\}K$.

b) The zero set $\Za(\Phi)$ belongs to the filter $\langle
\fH,\fG^{\bot*}\rangle$;

c) Each row $\Phi_h$ belongs to $K\{\fG^{\pr *}\}$. Furthermore, for every $A\in\fG$
the collection $\Phi^{t^*}$, $t\in\ov A$ is $\fH$-summable and

\begin{equation}\label{20}
\Phi\cdot \bigl(\suml_{t\in\ov A}{}^\th\delta^t k_t\bigl) =
\suml_{t\in\ov A}{}^{\fH}\Phi^{t^*}k_t
\end{equation}
for any set of coefficients $k_t\in K$.

d) the left multiplication by $\Phi$ defines a $\fG$-linear map $\phi$
from $\{\fG\}K$  to $\{\fH\}K$.

If $\phi$ is a linear continuous map from $\{\fG\}K$ to $\{\fH\}K$, then
its matrix $\Phi$ satisfies these equivalent conditions.
\end{theorem}
\begin{proof}
b) $\Ra$ a) follows from Proposition~\ref{prop:cont} and Lemma~\ref{prop:MM}.

a) $\Ra$ c). Since the product $\Phi\cdot \ba$ is defined for every $\ba\in \{\fG\} K$,
 by the definition of product of matrices we obtain $\Phi_h\in K\{\fG^{\pr *}\}$.

If $t\in \ov A$, then $\cZ(\del^t)\sep A$, hence
$\cZ(\del^t k_t)\sep A$ for any $k_t\in K$, and therefore
$\bigcap_{t\in \ov A}\cZ(\del^t k_t)\sep A$. It follows that
$\bigcap_{t\in \ov A}\cZ(\del^t k_t)\in \fG$, hence the family
$\del^t k_t$, $t\in \ov A$ is summable and clearly

$$\suml_{t\in \ov A}{}^{\th} \del^t k_t=\suml_{t\in \ov A}{}^{\fG} \del^t k_t\,.$$

Furthermore, by Theorem~\ref{sum} we have
$\suml_{t\in \ov A}{}^{\fG} \del^t k_t=\suml_{t\in \ov A}{}^{T(\fG)} \del^t k_t$.
Then
$$\phi\bigl[\suml_{t\in\ov A}{}^{\th}\delta^tk_t\bigl]=
\phi\bigl[\suml_{t\in\ov A}{}^{T(\fG)}\delta^tk_t\bigl]=
\suml_{t\in\ov A}{}^{T(\fH)}\ph[\delta^t]k_t\,.
$$

Since $\phi$ is continuous, the family $\Phi^{t^*} k_t$, $t\in \ov A$ is summable.
By Theorem~\ref{sum} and the previous equality, we obtain
$$
\Phi\cdot \bigl(\suml_{t\in\ov A}{}^\th\delta^t k_t\bigl)=
\phi \bigl[\suml_{t\in\ov A}{}^\th\delta^t k_t\bigl]=
\suml_{t\in\ov A}{}^{T(\fH)}\ph[\delta^t]k_t=
\suml_{t\in\ov A}{}^{T(\fH)}\Phi^{t^*}k_t=
\suml_{t\in\ov A}{}^{\fH}\Phi^{t^*}k_t
,
$$
as desired.

c) $\Ra$ d). Because $\Phi_h\in K\{\fG^{\pr *}\}$ for every $h\in H$,
the product $\Phi\cdot \ba$ is defined for every $\ba \in \{ \fG\} K$.
Furthermore, (\ref{20}) yields that $\Phi\cdot \ba \in \{\fH \} K$.
Indeed, if $\ba \in \{\fG\}K$, then $\cZ(\ba)\in \fG$. If $A=\cZ(\ba)$,
then $\ba = \suml_{t\in \ov A}{}^{\Tych} \del^t a_t$. By the assumption,
$$\Phi\cdot \ba =
\Phi\cdot \bigl(\suml_{t\in\ov A}{}^\th\delta^t a_t\bigl) =
\suml_{t\in\ov A}{}^{\fH}\Phi^{t^*} a_t\,.$$
Since the family $\Phi^{t^*}$, $t\in \ov A$ is $\fH$ summable, then clearly the family
$\Phi^{t^*}a_t$, $t\in \ov A$ is $\fH$-summable, therefore
$\Phi\cdot \ba=\suml_{t\in \ov A}{}^{\fH} \Phi^{t^*}a_t\in \fH$ (see a remark after
Definition~\ref{deff}).

Suppose that $\bb_j$, $j\in J$, where $\bb_j\in \{\fG\}K$, is a $\fG$-summable family.
Then
$A=\bigcap_{j\in J}\Za({\bb}_j)\in\fG$ and $\ba=\suml_{j\in J}{}^\fG{\bb}_j\in\{\fG\}K$.

From $\cZ(\bb_j)\sep A$ it follows that

\begin{equation}\label{tt3}
{\bb}_j=\suml_{t\in\ov A}{}^{\th}\,\delta^tk_{tj},
\end{equation}
where $j\in J$ and $k_{tj}=({\bb}_j)_t\in K$. Since the family $\bb_j$,
$j\in J$ is $\fG$-summable, for every $t\in \ov A$ there are only finitely
many $j\in J$ such that $k_{tj}\neq 0$. Thus we can set $k_t=\sum_{j\in J}
k_{tj}$ and then ${\ba}=\suml_{t\in\ov A}{}^\th\,\delta^t k_t$.

First we check that the family $\Phi\cdot {\bb}_j$, $j\in J$ is
$\fH$-summable. Indeed, by the assumption,
$\Phi\cdot {\bb}_j=\suml_{t\in\ov A}{}^{\fH} \Phi^{t^*} k_{tj}$,
therefore $\bigcap_{j\in J}\Za(\Phi\cdot {\bb}_j)\sep \bigcap_{t\in\ov
A}\Za(\Phi^{t^*})$. Because the family $\Phi^{t^*}$, $t\in\ov A$ is
$\fH$-summable, it follows that $\bigcap_{t\in \ov A}
\cZ(\Phi^{t^*})\in \fH$, hence
$\bigcap_{j\in J}\Za(\Phi\cdot {\bb}_j)\in\fH$.

It remains to prove that for every $h\in H$ there are only finitely
many $j\in J$ such that $(\Phi\cdot \bb_j)_h\neq 0$. Indeed, the family
$\Phi^{t^*}$, $t\in\ov A$
is $\fH$-summable, hence $\Delta=\{t\in\ov A \mid \Phi^{t^*}_h\neq 0\}$ is
a finite set. Furthermore, for any $t\in\ov A$ the set $J_t=\{j\in J
\mid k_{tj}\neq 0\}$ is also finite. Hence the set
$J_0:=\bigcup_{t\in\Delta} J_t$ is finite being a finite union of
finite sets. If $j\notin J_0$, then $j\notin J_t$ for any
$t\in\Delta$, and thus $k_{tj}=0$. Therefore,
$$
(\Phi\cdot {\bb}_j)_h=\suml_{t\in\ov A}\Phi^{t^*}_hk_{tj}=\suml_{t\in\ov
A\sm\Delta}0\cdot k_{tj}+ \suml_{t\in\Delta}\Phi^{t^*}_h\cdot0=0.
$$

Now we prove the equality
$\Phi\cdot \suml_{j\in J}{}^{\fG}{\bb}_j =
\suml_{j\in J}{}^{\fH}\Phi\cdot \bb_j.$ Indeed,
by the assumption we have
$\Phi\cdot \suml_{j\in J}{}^{\fG}{\bb}_j = \Phi\cdot \ba=
\suml_{t\in\ov A}{}^{\fH}\Phi^{t^*}k_t$. On the other hand

$$
\suml_{j\in J}{}^{\fH}\Phi\cdot {\bb}_j = \suml_{j\in
J}{}^{\fH}\suml_{t\in\ov A}{}^{\fH}\Phi^{t^*}k_{tj}= \suml_{t\in\ov
A}{}^{\fH}\Phi^{t^*}\suml_{j\in J}k_{tj}= \suml_{t\in\ov
A}{}^{\fH}\Phi^{t^*}k_t\,.
$$
Indeed, the first equality follows from (\ref{tt3}) and the
assumption; and the last equality follows from $k_t=\sum_j k_{tj}$.
The second equality will be checked coordinate-wise. Take any $h\in
H$. Then
$$\bigl(\suml_{j\in J}{}^{\fH}\suml_{t\in\ov A}{}^{\fH}\Phi^{t^*}k_{tj}\bigl)_h=
\suml_{j\in J}\Bigl(\suml_{t\in\ov A}{}\Phi^{t^*}k_{tj}\Bigl)_h\,,$$
where we omitted the superscript $\fH$ in the right hand part of the
equality, because the sum by $J$ in this part is finite. By the
definition of multiplication of matrices, this is equal to
$\suml_{j\in J}\suml_{t\in\ov A}\Phi^{t^*}_hk_{tj}$. We have already
proved that $\Phi^{t^*}_h k_{tj}\neq 0$ iff $t\in \Del$ and
$k_{tj}\neq 0$, where $\Del$ is a finite set. It follows that the set
of pairs $(j,t)$ such that $\Phi^{t^*}_h k_{tj}\neq 0$ is also
finite. Therefore me can change the summation order to get
$$
\suml_{j\in J}\Bigl(\suml_{t\in\ov A}{}\Phi^{t^*}k_{tj}\Bigl)_h=
\suml_{t\in\ov A}\Phi_h^{t^*}\suml_{j\in J}k_{tj}=
\Bigl(\suml_{t\in\ov A}{}^{\fH}\Phi^{t^*}\suml_{j\in
J}k_{tj}\Bigl)_h
$$

d) $\Ra$ b). Let $\ba=\suml_{g\in G}{}^\th\delta^gk_g\in \{\fG\}K$.
Since $\phi$ is $\fG$-linear, we obtain
$$
 \ph[{\ba}]=\suml_{g\in G}{}^{\fH}\ph[\delta^g]k_g=\suml_{g\in G}{}^{\fH}\Phi^{g^*}k_g
 =\Phi\cdot\ba.
$$

We will prove that $\Za(\Phi)\in\langle \fH,\fG^{\bot*}\rangle$.
First we check condition (\ref{eq:m2}): $B'=\bigcap_{t\in \ov{A}}\Za(\Phi^{t^*})\in \fH$
for every $A\in\fG$.

%Indeed, the inclusion $B'\times \ov{A^*}\seq \Za(\Phi)$  follows from the definition
%of $B'$.

Because the family $\delta^t$, $t\in \ov A$ is
$\fG$-summable, by the assumption, the family of columns
$\Phi^{t^*}$, $t\in \ov A$ is $\fH$-summable.  By Theorem~\ref{sum}
and the necessary condition of convergency (see
Section~\ref{S-sums}) we obtain
\[\liml_{\Cof(\ov{A})}{}^{T(\fH)}\Phi^{t^*}=0\,,\]
that is, the limit of the map $t\to \Phi^{t^*}$ from $\ov A$ to
${}^H K$ with respect to the Frechet filter on $\ov A$ is equal
to zero. This means that for any zero neighborhood $U(B,\fH)$,
$B\in\fH^{\bot}$ in $\{\fH\}K$ there exists a cofinite subset $A_0'$
of $\ov{A}$ such that $\Phi^{t^*}\in U(B,\fH)$ for any $t\in A_0'$.
In other words, if $t\in A_0'$, then $\Za(\Phi^{t^*})\sep \ov B$,
that is, $\Za(\Phi^{t^*})\cup B=H$.

Then
$$
B'\cup B=\bigl(\bigcap_{t\in \ov{A}}\Za(\Phi^{t^*})\bigl)\cup B=
\bigcap_{t\in\ov{A}}(\Za(\Phi^{t^*})\cup B)=
$$
$$
= \bigl(\bigcap_{t\in A_0'}(\Za(\Phi^{t^*})\cup B)\bigl) \cap
\bigl(\bigcap_{t\in \ov{A\cup A_0'}}(\Za(\Phi^{t^*})\cup B)\bigl)=
$$
$$
=H\cap\bigl(\bigcap_{t\in \ov{A\cup A_0'}}(\Za(\Phi^{t^*})\cup
B)\bigl) =\bigcap_{t\in\ov{A\cup A_0'}}(\Za(\Phi^{t^*})\cup B)\,.
$$
Note that the last set is cofinite in $H$. Indeed, $\Phi^{t^*}\in \{\fH\}K$
and $B\in \fH^\pr$ implies that each $\Za(\Phi^{t^*})\cup B$ is cofinite, and
$\ov{A\cup A'_0}=\ov A\sm A_0'$ is a finite set.

Thus we proved that the union $B'\cup B$ is cofinite for any
$B\in\fH^\bot$. This means that $B'\in\fH^{\bot\bot}=\fH$, as
required.

Now we have to check (\ref{eq:m1}):
$$
A':=\bigcap_{s\in\ov B}\Za(\Phi_s)\in\fG^{\bot*}\ \text{ for every } \
B\in \fH^\bot\,.
$$
 This is the same as the
union $A^*\cup A'$ is cofinite for any $A\in \fG$. We have the
following sequence of equivalences:
$$
t\in A'\LR \forall s\in \ov B\,\,(t\in\Za(\Phi_s))\LR\forall s\in
\ov B\,\,(\Phi_s^t=0)\LR
$$
$$
\LR \forall s\in \ov B\,\,(s\in\Za(\Phi^t))\LR \Za(\Phi^t)\sep \ov
B\LR \Za(\Phi^t)\cup B=H.
$$
Hence:
$$A'=\set{t\in G\mid \Za(\Phi^t)\cup B=H}.$$
We have already proved that there is a cofinite subset $A_0'$ of
$\ov A$ such that $\Za(\Phi^{t^*})\cup B=H$ for all $t\in A_0'$.
Then $A^{\prime*}_0\subseteq A'$. Since $A\cup A'_0$ is cofinite in $H$,
the same is true for $A^*\cup A^{\prime*}_0$. From
$A^{\prime*}_0\seq A'$ it follows that $A^*\cup A'$ is also
cofinite in $H$, as desired.
\end{proof}

If $X$ is a $K$-linear topological space, then a map
$\phi: X\to K$ is said to be a \emph{linear form}, if $\phi$ is linear
and continuous. This space can be endowed with a weak topology:
a net of linear forms $\phi_i: X\to K$, $i\in I$ converges to $\phi$,
if $\phi_i(x)$ converges to $\phi(x)$ for every $x\in X$.

\begin{cor}\label{cor}
Suppose that $\fG$ is a balanced filter on $G$. Then the space of
continuous linear forms of the topological linear space of columns
$(\{\fG\}K,T(\fG))$ is isomorphic to the topological linear space of
rows $(K\{\fG^{\bot*}\}, T(\fG^{\bot*}))$
\end{cor}
\begin{proof}
Applying Theorem~\ref{thm:endo} to the case when $H$ consists of
one element, we obtain that every continuous linear map
from $\{\fG\}K$ to $K$ is given by a left multiplication by a row
from $K\{\fG^{\bot*}\}$.
\end{proof}

For further applications we need a dual variant of Theorem~\ref{thm:endo}.

\begin{remark}\label{rem}
Suppose that $\Psi$ is an $H\times G$ matrix.
Then the following are equivalent:

a) the right multiplication by $\Psi$ defines a linear continuous map
from $K\{\fG\}$ to $K\{\fH\}$;

b) $\cZ(\Psi)\in \lan \fG^{\pr *}, \fH \ran$.

If $\psi$ is a continuous linear map from $K\{\fG\}$ to $K\{\fH\}$, then its matrix
$\Psi$ satisfies these equivalent conditions.
\end{remark}

\section{The ring of continuous operators}\label{S-ring}

Suppose that $\Phi$ is an $H\times G$ matrix, $\ga$ is a row and $\ba$ is a
column such that  the products $\ga\cdot\Phi$ and $\Phi\cdot \ba$ are defined.

\begin{remark}
$$
\Za(\Phi\cdot \ba)\sep\bigcap_{t\in G\sm\Za(\ba)}\Za(\Phi^{t^*})
\quad \text{ and } \quad
\Za(\ga\cdot\Phi)\sep\bigcap_{s\in H\sm\Za(\ga)}\Za(\Phi_{s^*}).
$$
\end{remark}
\begin{proof}
Indeed, if $h\in\bigcap_{t\in G\sm\Za(\ba)}\Za(\Phi^{t^*})$, then

\[(\Phi\cdot \ba)_h=\suml_{t\in G}\Phi_h^{t^*} a_t=
\suml_{t\in \Za(\ba)}\Phi_h^{^*}\cdot 0 + \suml_{t\in
G\sm\Za(\ba)}\Phi_h^{t^*} a_t =0+\suml_{t\in G\sm\Za(\ba)}0\cdot
a_t=0.\]

The second inclusion has a similar proof.
\end{proof}

Now we prove an auxiliary lemma.

\begin{lemma}\label{prpty:04}
For any $\beta\in K\{\fG^{\bot*}\}$ and $\bb\in\{\fH\}K$ the matrix
$\Phi=\bb\cdot\beta\in{}^HK^G$ satisfies (\ref{eq:m1}) and
(\ref{eq:m2}). Therefore (by Proposition~\ref{prop:cont}) the left multiplication
by $\Phi$ defines a continuous linear map from $\{\fG\}K$ to $\{\fH\} K$.
\end{lemma}
\begin{proof}
First we check (\ref{eq:m1}). Take any $B\in\fH^\bot$. If $s\in H$, then
$(\bb\cdot\beta)_s=b_s\beta$,
where $b_s\in K$ is the $s$th coordinate of $\bb$. If $b_s=0$, then
$\Za((\bb\cdot\beta)_s)=G$, and otherwise
$\Za((\bb\cdot\beta)_s)=\Za(\beta)$. Therefore,

\[\bigcap_{s\in \ov B}\Za((\bb\cdot\beta)_s)\supseteq\Za(\beta)\in\fG^{\bot*},\]
which implies $\bigcap_{s\in \ov
H}\Za((\bb\cdot\beta)_s)\in\fG^{\bot*}$, as desired. The proof of
(\ref{eq:m2}) is similar.
\end{proof}

Recall that in Section~\ref{S-matrix} we gave an example that a product of
matrices is not associative. It will become associative if we put
some extra restrictions.

\begin{prop}\label{prpty:06}
Suppose that  $\Upsilon$, $\Phi$ and $\Psi$ are matrices of the
following size: $\Upsilon\in{}^JK^H$, $\Phi\in{}^HK^G$,
$\Psi\in{}^GK^I$. Further assume that

a) for any pair $(j,g)$ there exist only finitely many $h\in H$
such that $\Upsilon_j^{h^*}\Phi_h^g\neq 0$;

b) for any pair $(h,i)$  there exist only finitely many $g\in G$
such that $\Phi_h^{g^*}\Psi_g^i\neq 0$;

c) for any pair $(j,i)\in J\times I$ the exist only finitely many
$(h,g)\in H\times G$ such that
$\Upsilon_j^{h^*}\Phi_h^{g^*}\Psi_g^i\neq0$;

Then the products $\Upsilon\cdot(\Phi\cdot\Psi)$ and
$(\Upsilon\cdot\Phi)\cdot\Psi$ are defined and equal.
\end{prop}
\begin{proof}
Clearly a) and b) is nothing more that the existence of the products
$\Upsilon\cdot\Phi$ and $\Phi\cdot\Psi$. From c) we obtain the
following equality:

\[\suml_{(h,g)\in H\times G}\Upsilon_j^{h^*}\Phi_h^{g^*}\Psi_g^i=
\suml_{h\in H}\Upsilon_j^{h^*}\bigl(\suml_{g\in G}\Phi_h^{g^*}\Psi_g^i\bigl)=
\suml_{g\in G}\bigl(\suml_{h \in
H}\Upsilon_j^{h^*}\Phi_h^{g^*}\bigl)\Psi_g^i\,,\] which implies
associativity.
\end{proof}

This proposition applies in the following situation.

\begin{prop}\label{prpty:07}
Suppose that $\bb_1$, $\bb_2$, $\ldots$ are columns from $\{\fG\}K$, and
$\beta_1,\beta_2,\ldots$ are rows from $K\{\fG^{\bot*}\}$. Then
any (finite) alternating product
$\bb_1\cdot\beta_1\cdot \bb_2\cdot\beta_2 \cdot\ldots$ or
$\beta_1 \cdot \bb_1\cdot\beta_2\cdot \bb_2\cdot \dots$ does not
depend on the way we put brackets on it.
\end{prop}
\begin{proof}
We prove only associativity for short products:
\[(\bb\cdot\beta)\cdot \bc=\bb\cdot(\beta\cdot \bc)\quad \text{ and } \quad
(\beta\cdot \bb)\cdot\ga=\beta(\cdot \bb\cdot\ga)\]
for any columns $\bb,\bc$ and any rows $\beta,\ga$. The general case is derived by
induction as in \cite[Sec.~1]{Lang}.

Clearly the conditions a) and b) from Proposition \ref{prpty:06} are satisfied.
Furthermore, c) of this proposition means the following: for every
$j\in G$ there are only finitely many $g^*\in G$ such that
$b_j\beta^{g^*} c_g\neq 0$. But this follows from the fact that $g$
belong to $(\supp\beta)^*\cap\supp \bc$, which is a finite set.

The proof of the second equality is similar.
\end{proof}

Now we define a ring  $\MOG$ of `finitary' $G\times G$-matrices.
Later we will show that this ring can be considered as a
dense subring in the algebra of all continuous linear operators on
$\{\fG\} K$.

Let $\MOG$ denote the set of all $G\times G$-matrices
\[
\ba_1\ga_1+\ba_2\ga_2+\ldots+\ba_n\ga_n,
\]
where $\ba_i\in\{\fG\}K$ and $\ga_i\in K\{\fG^{\bot*}\}$, $i=1,
\dots, n$.

\begin{prop}\label{prpty:08}
The set $(\MOG,+,\cdot)$ is an associative ring.
\end{prop}
\begin{proof}
Clearly this set is closed with respect to addition.

We show that the product of elements of $(\MOG,+,\cdot)$ is an
element of $(\MOG,+,\cdot)$. Clearly it suffices to check this for
elements of the form $\ba\cdot \ga$ and $\bb\cdot \beta$, where
$\ba,\bb$ are columns and $\ga,\beta$ are rows. Indeed,
$\ga\cdot \bb=k \in K$. Applying Proposition~\ref{prpty:07}, we obtain

\[(\ba\cdot\ga)\cdot(\bb\cdot\beta)=\ba(\ga\cdot \bb)\beta=(\ba k)\beta=\ba(k\beta)\,.\]

It remains to use distributivity (see Lemma~\ref{dist}).
\end{proof}

The following proposition also claims an associativity.

\begin{prop}\label{prop:Assoc}
Suppose that $\fG$ and $\fH$ are balanced filters on $G$ and $H$.
Further assume that $\ga\in K\{\fH^{\bot*}\}$ is a row,
$\ba\in\{\fG\}K$ is a column, and $\Phi\in {}^HK^G$ is an $H\times
G$-matrix such that $\cZ(\Phi)\in \lan \fH, \fG^{\pr *}\ran$. Then
$(\ga\cdot\Phi)\cdot \ba=\ga\cdot(\Phi\cdot \ba)$.
\end{prop}
\begin{proof}
From Theorem \ref{top-UA} 3) it follows that
$\ba=\liml_i{}^{T(\fG)}\,\ba_i$, where $\ba_i$ are columns with
finite support. If $i$ is fixed, then $(\ba_i)_g\neq 0$ for finitely
many $g\in G$. For any such $g$, there exists only finitely many $h$
such that $\ga^{h^*} \Phi^{g^*}_h\neq 0$. Thus there are only
finitely many pairs $(g, h)$ such that $\ga^{h^*} \Phi^{g^*}_h
(\ba_i)_g\neq 0$. This means that hypotheses of
Proposition~\ref{prpty:06} are satisfied.
Now we have
$$(\ga\cdot\Phi)\cdot \ba=(\ga\cdot\Phi)\cdot\lim_i{}^{T(\fG)}\ba_i\,.$$
Furthermore, by the assumption (and the dual variant of
Proposition~\ref{prop:cont}) the right multiplication by $\Phi$ acts
as a linear continuous operator from $K\{\fH^{\pr *}\}$ to $K\{\fG^{\pr *}\}$. In
particular, $\ga\cdot \Phi\in K\{\fG^{\pr *}\}$. Then, by
Corollary~\ref{cor}, the left multiplication by $\ga\cdot \Phi$ is a
linear continuous map from $\{\fG\} K$ to $K$. Thus

$$(\ga\cdot\Phi)\cdot\lim_i{}^{T(\fG)}\ba_i=
\liml_i{}^{\Disc}(\ga\cdot\Phi)\cdot \ba_i\,,
$$
where $\Disc$ denotes the discrete topology on $K$. By
Proposition~\ref{prpty:08} we obtain
$(\ga\cdot\Phi)\cdot \ba_i=\ga\cdot(\Phi\cdot \ba_i)$, therefore

$$
\liml_i{}^{\Disc}(\ga\cdot\Phi)\cdot \ba_i=
\liml_i{}^{\Disc}\ga\cdot(\Phi\cdot \ba_i)\,.
$$

Since left multiplications by $\ga$ and $\Phi$ are continuous operators
(see Theorem~\ref{thm:endo}), we conclude that

$$\liml_i{}^{\Disc}\ga\cdot(\Phi\cdot \ba_i)= \ga\cdot\liml_i{}^{T(\fH^\bot)}(\Phi\cdot
\ba_i)=\ga\cdot(\Phi\cdot\liml_i{}^{T(\fG)}\ba_i)\,.$$
Because $\ba =\lim_i{}^{T(\fG)}\ba_i$, this is equal to $\ga\cdot (\Phi \cdot
\ba)$, as desired.
\end{proof}

Now we give a matrix description of the algebra of linear continuous operators of
the column space $\{\fG\} K$.

\begin{theorem}\label{thm:Ring}
Let $\fG$ be a balanced filter on a set $G$ with involution, and
let $\MG$ be a set of all $G\times G$-matrixes $\Phi$ such that
$\Za(\Phi)\in\lan \fG,\fG^{\bot*}\ran$. Then

a) $\MG$ is a ring with respect the above defined addition
and multiplication;

b) If $\Phi\in \MG$, then the map $\wh\Phi: \ba \to \Phi\cdot \ba$
is a continuous endomorphism of the column space $\{\fG\}K$.

c) The map $\Phi\to\hat\Phi$ is an isomorphism from $\MG$ onto the
ring of all continuous endomorphisms of the space $\{\fG\}K$.
%The same properties take place for the rows.
\end{theorem}

\begin{proof}
b) follows from the equivalence of a) and b) in Theorem
\ref{thm:endo}.

a), c) First we prove that $\MG$ is closed with respect to
multiplication. Suppose that $\Phi, \Psi\in \MG$. Since $\cZ(\Phi),
\cZ(\Psi)\in \lan \fG,\fG^{\bot*}\ran$, by Remark~\ref{prop:ex} a) and its dual variant,
we obtain that $\cZ(\Phi_g)\in \fG^{\pr *}$ and $\cZ(\Psi^{h^*})\in \fG$ for all
$g, h\in G$. It follows that the product $\Phi\cdot \Psi$ is defined.

Take any $g\in G$. By Proposition~\ref{prop:Assoc} we obtain
$(\Phi_g \cdot \Psi)\cdot \ba = \Phi_g \cdot (\Psi\cdot \ba)$.
Then by the definition of matrix multiplication we derive
$$(\Phi_g \cdot \Psi)\cdot \ba =(\Phi \cdot \Psi)_g \cdot \ba=
((\Phi \cdot \Psi)\cdot \ba)_g\,,$$ and $\Phi_g \cdot (\Psi\cdot
\ba)=(\Phi \cdot (\Psi \cdot \ba))_g$, therefore $(\Phi \cdot
\Psi)\cdot \ba= \Phi \cdot (\Psi \cdot \ba )$. Because the left
multiplication by $\Phi$ and $\Psi$ is continuous, this equality
implies that the left multiplication by $\Phi\cdot \Psi$ is
continuous. By Theorem~\ref{thm:endo} we conclude that $\Phi\cdot
\Psi\in \MG$.

Furthermore, writing $(\Phi \cdot \Psi)\cdot \ba= \Phi \cdot (\Psi
\cdot \ba )$ at the level of endomorphisms, we obtain $\wh{\Phi
\cdot \Psi}=\wh\Phi\circ \wh \Psi$, therefore the map $\Phi\to \wh
\Phi$ preserves multiplication. Clearly this map preserves an
addition, hence it is a morphism of rings (associative or not).

By Theorem~\ref{thm:endo}, if $\phi$ is a continuous linear operator
of the space $\{\fG\} K$, then there is a unique matrix $\Phi\in
\MG$ such that $\phi[\ba]=\Phi\cdot \ba$. It follows that the map $\Phi\to \wh \Phi$
is an isomorphism of rings, in particular $\MG$ is an associative ring.
\end{proof}

Suppose that $\fG$, $\fH$ are balanced filters on sets with involution $G$ and $H$.
Let $\La(\fH,\fG)$ be the space of all linear continuous maps $\{\fG\}K\to\{\fH\}K$.
By Theorem~\ref{thm:endo} we may identify $\La(\fH,\fG)$ with the space of
$H\times G$-matrices $\Phi$ such that $\Za(\Phi)\in\lan \fH,\fG^{\bot*}\ran$. Therefore
we may consider $\La(\fH,\fG)$ as a topological space with the topology
$T(\langle\fH,\fG^{\bot*}\rangle)$ defined by the filter
$\lan \fH,\fG^{\bot*}\ran$ on $H\times G$.

It is well known that for normed linear spaces, the convergency
of a net of continuous operators with respect to the operator norm
implies their strong convergency. The following proposition gives a version of this
result in our situation.

\begin{prop}\label{cc}
Suppose that a net of  $H\times G$-matrices $\Phi_i\in\La(\fH,\fG)$ converges
to $\Phi$ with respect to $T(\langle\fH,\fG^{\bot*}\rangle)$. Then for any column
$\ba\in\{\fG\}K$ and for any row $\ga\in K\{\fH^{\bot*}\}$
\[\liml_i{}^{T(\fG)}\Phi_i\cdot\ba=\Phi\cdot\ba \ \text{ and } \
\liml_i{}^{T(\fH^{\bot*})}\ga\cdot\Phi_i=\ga\cdot\Phi.\]
\end{prop}
\begin{proof}
Clearly we may assume that $\Phi=0$, hence $\Phi\cdot\ba=\mathbf{0}$. Thus for the
first equation we have to prove that $\liml_i{}^{T(\fG)}\Phi_i\cdot\ba=0$.

Recall that by Theorem~\ref{ht} we have
$\langle\fH,\fG^{\bot*}\rangle^\bot=\fH^\bot\otimes\fG^*$.
Furthermore (see Section~\ref{S-Space}) the basis of the filter
$\fH^{\pr}\otimes \fG^*$ is
given by the sets $\ov{\ov B_1\times \ov A_1}$, where $B_1\in \fH$ and $A_1\in \fG^*$.

Let $U(B,\fH)$, $B\in\fH^\bot$ be a zero neighborhood of the space $\{\fH\}K$.
We have to find a zero neighborhood
$U=U(\ov{\ov B_1\times \ov A_1},T(\lan \fH,\fG^{\bot*}\ran))$, of the space $\La(\fH,\fG)$
such that $\Phi_i\in U$ implies $\Phi_i\cdot\ba\in U(B,\fH)$.

Take $B_1=B$ and $A_1=\Za(\ba)^*$. If $\Phi_i\in U$ then, by the definition $U$, we have
\[\Za(\Phi_i)\supseteq\ov B_1\times\ov A_1=\ov B\times\ov{\Za(\ba)}^*.\]
If $h$ is an arbitrary element of $\ov B$, then
\[(\Phi_i\cdot\ba)_h=\suml_{g\in G}(\Phi_i)_h^{g^*}\ba_g=\suml_{g\in \ov{\Za(\ba)}}0\cdot\ba_g+
\suml_{g\in \Za(\ba)}(\Phi_i)_h^{g^*}\cdot 0=0.\]
It follows that $\Za(\Phi_i\cdot\ba)\supseteq\ov B$, that is,
$\Phi_i\cdot\ba\in U(B,\fH)$, as desired.

The second equality can be verified similarly.
\end{proof}

By Lemma~\ref{prpty:04}, $\MOG$ is a subring of $\MG$. The following remark
shows that this subring is dense.

\begin{remark}
$\MOG$ is a dense subring of $\MG$ with respect to topology $T(\lan \fG, \fG^{\pr *}\ran)$.
\end{remark}
\begin{proof}
The set $K[H\times G]$ of all matrices with finite support forms a subalgebra of the
algebra $\MOG$. By Theorem~\ref{top-UA}, this algebra is dense in $\MG$. It remains
to notice that $K[H\times G]\seq \MOG$, because every finite matrix is a linear
combination of matrices $\del^g\cdot \del_h\in \MOG$.
\end{proof}

\end{document}